\documentclass[11pt,letterpaper]{amsart}
\usepackage{amsfonts}
\usepackage{amsthm}
\usepackage{amsmath}
\usepackage{amscd}
\usepackage[latin2]{inputenc}
\usepackage{t1enc}
\usepackage[mathscr]{eucal}
\usepackage{indentfirst}
\usepackage{graphicx}
\usepackage{graphics}
\usepackage{enumitem}
\setenumerate[1]{label=\thesection.\arabic*.}
\setenumerate[2]{label*=\arabic*.}
\usepackage{url}
\usepackage{pict2e}
\usepackage{pgfplots}
\usepackage{epic}
\numberwithin{equation}{section}
\usepackage[left=1in, right=1in, top=1.0in, bottom=1.4in]{geometry}
\usepackage{epstopdf} 
\usepackage{tikz-cd}
\usepackage{lipsum}
\usepackage{hyperref}
\usepackage{mathrsfs}

\hypersetup{colorlinks,linkcolor={blue},citecolor={blue},urlcolor={gold}}
\newtheorem{theorem}{Theorem}[section]
\newtheorem{corollary}[theorem]{Corollary}
\newtheorem{lemma}[theorem]{Lemma}
\newtheorem{proposition}[theorem]{Proposition}

\theoremstyle{definition}
\newtheorem{definition}[theorem]{Definition}
\newtheorem{example}[theorem]{Example}

\makeatletter
\newcommand*\bigcdot{\mathpalette\bigcdot@{.5}}
\newcommand*\bigcdot@[2]{\mathbin{\vcenter{\hbox{\scalebox{#2}{$\m@th#1\bullet$}}}}}
\makeatother

\newlist{steps}{enumerate}{1}
\setlist[steps, 1]{label = Step \arabic*.}

\title{\vspace{-6cm}The Holomorphic Extension Property for Higher Du Bois Singularities}

\author[B.Tighe]{Benjamin Tighe}
\address{Department of Mathematics, University of Oregon}
\email{bentighe@uoregon.edu}

\begin{document}

\maketitle 

\begin{abstract}
Let $X$ be a normal complex algebraic variety.  We show that the holomorphic extension property holds in degree $p < \mathrm{codim}_X(X_{\mathrm{sing}})$ when $X$ has Du Bois singularities, giving an improvement on Flenner's criterion for arbitrary singularities.  As an application, we study the $m$-Du Bois definition from the perspective of holomorphic extension and compare how different restrictions on $\mathscr H^0(\underline \Omega_X^p)$ affect the singularities of $X$, where $\underline\Omega_X^p$ is the $p^{th}$-graded piece of the Du Bois complex.
\end{abstract}

\section{Introduction}

\subsection{Holomorphic and Logarithmic Extension} Let $X$ be a normal complex variety with regular locus $U$ with inclusion morphism $j:U \hookrightarrow X$.  We say that the \textit{holomorphic extension property} holds in degree $p$ if the natural inclusion \begin{equation} \label{equation holomorphic extension}
    \pi_*\Omega_{\widetilde X}^p \hookrightarrow \Omega_X^{[p]} : = j_*\Omega_U^p
\end{equation} is an isomorphism for some --- and therefore any --- resolution of singularities $\pi: \widetilde X \to X$.

The holomorphic extension property has been extensively studied for the classes of singularities arising in the minimal model program.  For klt and rational singularities, the holomorphic extension property holds for every $p$, see \cite[Theorem 1.4]{greb2011differential} and \cite[Corollary 1.8]{kebekus2021extending}.  Rational singularities consequently satisfy many important properties:

\begin{itemize}
    \item Functorial pullback for the sheaves of reflexive differentials $\Omega_X^{[p]}$ \cite[Theorem 1.3]{kebekuspullback}, \cite[Theorem 1.11]{kebekus2021extending}.

    \item The sheaf $\Omega_{X,h}^p$ of $h$-differential $p$-forms agrees with $\Omega_X^{[p]}$ \cite[Thm. 1]{huber14}, \cite[Corollary 1.12]{kebekus2021extending}.

    \item The Zariski-Lipman conjecture holds for rational singularities: given a normal complex analytic variety with rational singularities and locally free tangent sheaf, then $X$ is smooth \cite[\S 6]{greb2011differential}, \cite[Theorem 1.14]{kebekus2021extending}.
\end{itemize}

Holomorphic extension is a weak condition for small $p$: if $\Sigma$ is the singular locus of $X$, then (\ref{equation holomorphic extension}) is an isomorphism for every $0 \le p < \mathrm{codim}_X(\Sigma) - 1$ by Flenner's criterion \cite[Theorem]{flenner88extendability}.  Moreover, if holomorphic extension holds in degree $p$, then holomorphic extension holds in degree $k$ for every $0 \le k \le p$ \cite[Theorem 1.4]{kebekus2021extending}.  For $p\ge \mathrm{codim}_X(\Sigma) - 1$, extension is less topological and (\ref{equation holomorphic extension}) can be strict: by definition, Gorenstein log-canonical singularities fail holomorphic extension in degree $\dim X$.  Instead, it is better to consider a variant of the inclusion (\ref{equation holomorphic extension}) which allows for logarithmic poles.  We say that a normal variety $X$ satisfies the \textit{logarithmic extension property} in degree $p$ if the inclusion \begin{equation}\label{equation logarithmic extension}
    \pi_*\Omega_{\widetilde X}^p(\log E) \hookrightarrow \Omega_X^{[p]}
\end{equation} is an isomorphism for some --- and therefore any --- log-resolution of singularities $\pi:\widetilde X \to X$ with log-exceptional divisor $E = \pi^{-1}(\Sigma)$.  Log-canonical singularities satisfy the logarithmic extension property for all $0 \le p \le \dim X$ \cite[Theorem 1.5]{greb2011differential}.  In general, logarithmic extension in degree $p$ implies logarithmic extension in degree $k$ for every $0 \le k \le p$ \cite[Theorem 1.5]{kebekus2021extending}.

\subsection{Extension for Du Bois Singularities} Let $X$ be a complex algebraic variety.  We say that $X$ has \textit{Du Bois singularities} if the natural morphism \begin{equation}\label{equation Du Bois definition intro}
    \mathscr O_X \to \underline \Omega_X^0
\end{equation} is a quasi-isomorphism, where $\underline \Omega_X^0 : = \mathrm{gr}^F_0\underline \Omega_X^\bullet$ is the $0^{th}$-graded piece of the \textit{Du Bois complex} $(\underline \Omega_X^\bullet, F)$, an object in the category of filtered complexes of $\mathscr O_X$-modules generalizing the holomorphic de Rham complex of smooth algebraic varieties.  

Du Bois singularities are an important class of singularities in the Hodge theory of singularities, as they arise in the construction of Deligne's mixed Hodge structure \cite{deligne1974theorie} on the cohomology of algebraic varieties.  They also play a fundamental role in the minimal model program, as log-canonical and rational singularities are Du Bois, see \cite[Theorem 1.4]{kollar2010log} \cite[Corollary 2.6]{kovacs1999rational}, \cite[5.4. Theorem]{saito2000mixed}. We consider then the following questions:

\begin{enumerate}[label=\normalfont(\roman*)]
    \item \label{question Du Bois holomorphic} If $X$ has Du Bois singularities, for which $p$ is the inclusion morphism (\ref{equation holomorphic extension}) an isomorphism?

    \item \label{question Du Bois logarithmic} If $X$ has Du Bois singularities, is the inclusion morphism (\ref{equation logarithmic extension}) an isomorphism for all $p$?
\end{enumerate}

If $X$ is Du Bois and Cohen-Macaulay, the answer to Question \ref{question Du Bois logarithmic} is already known: if $\pi:\widetilde X \to X$ is a log-resolution of singularities with exceptional divisor $E$, then the sheaf $\pi_*\omega_{\widetilde X}(E)$ is reflexive by \cite[Theorem 1.1]{kovacs2010canonical}\footnote{In fact, a Cohen-Macaulay variety is Du Bois if and only if $\pi_*\omega_{\widetilde X}(E)$ is reflexive, giving a natural generalization of Kempf's criterion for rational singularities.}, and so $\pi_*\Omega_{\widetilde X}^p(\log E)$ is reflexive for all $0 \le p \le \dim X$ \cite[Theorem 1.5]{kebekus2021extending}.  This gives a short proof of a result of Kov\'acs-Graf.

\begin{theorem}{\cite[Theorem 4.1]{graf2014potentially}} \label{theorem logarithmic extension intro}
If $X$ is a normal complex algebraic variety with at worst Du Bois singularities, then logarithmic extension holds in all degrees $0 \le p \le \dim X$: for any logarithmic resolution of singularities $\pi:\widetilde X \to X$, the inclusion $\pi_*\Omega_{\widetilde X}^p(\log E) \hookrightarrow \Omega_X^{[p]}$ is an isomorphism.  
\end{theorem}

Theorem \ref{theorem logarithmic extension intro} is key to understanding for which $p$ holomorphic extension holds.  Our main result is an extension of Flenner's criterion to Du Bois singularities and an optimal answer to Question \ref{question Du Bois holomorphic}.

\begin{theorem} \label{theorem holomorphic extension intro}
If $X$ is a normal complex algebraic variety with at worst Du Bois singularities and singular locus $\Sigma$, then holomorphic extension holds in degree $0 \le p < \mathrm{codim}_X(\Sigma)$: if $\pi:\widetilde X \to X$ is any resolution of singularities, the inclusion morphism $\pi_*\Omega_{\widetilde X}^p \hookrightarrow \Omega_X^{[p]}$ is an isomorphism.  
\end{theorem}

\subsection{Proof of Theorem \ref{theorem holomorphic extension intro} for Isolated Singularities}\label{subsection intro proof 1.2} We note that Theorem \ref{theorem holomorphic extension intro} (and its dependence on Theorem \ref{theorem logarithmic extension intro}) has been observed in the literature in special cases. For one, Graf-Kov\'acs observe the Zariski-Lipman conjecture holds for Du Bois singularities by demonstrating that the natural inclusion \begin{equation}\label{equation inclusion holomorphic log} \pi_*\Omega_{\widetilde X}^p \hookrightarrow \pi_*\Omega_{\widetilde X}^p(\log E)  
\end{equation} is an isomorphism for $p = 1$, where $\pi:\widetilde X \to X$ is a log-resolution of singularities, and showing $\pi_*\Omega_{\widetilde X}^1(\log E)$ is reflexive \cite{graf2013optimal}.  By Flenner's criterion, this is novel exactly when $\mathrm{codim}_X(\Sigma) = 2$. Theorem \ref{theorem holomorphic extension intro} can also be seen to hold as a corollary of Theorem \ref{theorem logarithmic extension intro} for \textit{isolated singularities} by an old result of van Straten-Steenbrink: if $j:U \hookrightarrow X$ is the inclusion of the regular locus $U$, there is a differentiation morphism $$d: j_*\Omega_U^{\dim X-1}/\pi_*\Omega_{\widetilde X}^{\dim X - 1} \to j_*\omega_U/\pi_*\omega_{\widetilde X}(E)$$ for any log-resolution of singularities $\pi:\widetilde X \to X$, which is injective if $X$ has isolated singularities \cite[Cor. 1.4]{van1985extendability}. 

An obvious idea is to reduce Theorem \ref{theorem holomorphic extension intro} to the case of isolated singularities by cutting $X$ down by successive hyperplane sections.  More specifically, let $\pi:\widetilde X \to X$ be a log-resolution of singularities with exceptional divisor $E$.  If $H$ is a very general hyperplane section of $X$, then there is an induced log-resolution of singularities $\pi|_H:\widetilde H \to H$, and we let $E|_H$ be the induced exceptional divisor.  For $p \ge 1$, consider the commutative diagram 
\begin{equation}\label{equation conormal form diagram}
\begin{tikzcd} 
   N_{H|X}^*\otimes (\pi|_H)_*\Omega_{\widetilde H}^{p-1} \arrow{r} \arrow{d} & \mathscr O_H \otimes \pi_*\Omega_{\widetilde X}^p \arrow{r} \arrow{d} & (\pi|_H)_*\Omega_{\widetilde H}^p \arrow{d} \\
   N_{H|X}^*\otimes (\pi|_H)_*\Omega_{\widetilde H}^{p-1}(\log E|_H) \arrow{r} & \mathscr O_H \otimes \pi_*\Omega_{\widetilde X}^p(\log E) \arrow{r} & (\pi|_H)_*\Omega_{\widetilde H}^p(\log E|_H)
\end{tikzcd}
\end{equation} Recalling that general hyperplane sections preserve the Du Bois property, an inductive hypothesis would imply that the left and right vertical morphisms are isomorphisms --- this is exactly the approach used in \cite{graf2013optimal} for $p = 1$, but they need to use the negativity lemma \cite[Proposition 7.5]{greb2010extension}, which seems particularly special for $p = 1$.  Induction seems insufficient in proving $\mathscr O_H\otimes \pi_*\Omega_{\widetilde X}^p \to \mathscr O_H\otimes \pi_*\Omega_{\widetilde X}^p(\log E)$ is an isomorphism, and the reflexivity of the sheaf $(\pi|_H)_*\Omega_{\widetilde H}^p(\log E|_H)$ is not sufficient to prove the reflexivity of $\pi_*\Omega_{\widetilde X}^p(\log E)$.  We require additional input.

\subsection{Extension Criterion and Hodge Modules}

To prove Theorem \ref{theorem holomorphic extension intro}, we consider the following well-known interpretation of extension for coherent sheaves:

\begin{proposition} \label{proposition kebekus sheaf criterion}\cite[Corollary 6.2]{kebekus2021extending}
Let $Y$ be a complex manifold and $\mathscr F$ a coherent sheaf of $\mathscr O_Y$-modules.  If $\mathrm{Supp}~\mathscr F$ has pure dimension $n$, the following are equivalent:

\begin{enumerate}[label=\normalfont(\roman*)]
    \item Section of $\mathscr F$ extend uniquely across any subset $A \subset Y$ with $\dim A \le n-2$.

    \item For every $k \ge -n + 1$ $\dim \mathrm{Supp}~R^k\mathscr Hom_{\mathscr O_Y}(\mathscr F, \omega_Y^\bullet) \le -(k + 2$), where $\omega_Y^\bullet$ is the dualizing complex of $Y$.
\end{enumerate}
\end{proposition}

We will often apply this to a singular variety $X$ of dimension $n$ by considering a \textit{local} embedding $X \subset Y$ into a smooth complex manifold $Y$. There is also a derived version of this criterion discussed in \cite{kebekus2021extending}.  We remark that this statement is stronger than holomorphic extension:

\begin{proposition}\label{proposition kebekus criterion complexes}
\cite[Proposition 6.4]{kebekus2021extending} Let $Y$ be a complex manifold, let $A \subset Y$ be a complex subspace, and let $K \in \mathrm{D}_{\mathrm{coh}}^b(\mathscr O_Y)$ be a complex with $\mathscr H^jK = 0$ for $j < 0$.  If $$\dim (A \cap \mathrm{Supp}~R^k\mathscr Hom_{\mathscr O_Y}(K, \omega_Y^\bullet)) \le -(k + 2)$$ for every $k \in \mathbb Z$, then the sections of $\mathscr H^0K$ extend uniquely across $A$.
\end{proposition}

To highlight the idea, let $X$ be a normal variety of dimension $n$ with at worst \textit{isolated} Du Bois singularities, which implies that the cohomology sheaves $R^j\pi_*\mathscr O_{\widetilde X}(-E) = 0$ for $i > 0$ for any log-resolution of singularities $\pi:\widetilde X \to X$.  We use Proposition \ref{proposition kebekus criterion complexes} to test holomorphic extension on \begin{equation}\label{equation log Grauert-Riemenschneider}
    \pi_*\omega_{\widetilde X}(E) = R^0\pi_*\omega_{\widetilde X}(E).
\end{equation} By Grothendieck duality, $$R^j\mathscr Hom_{\mathscr O_X}(\mathbf R\pi_*\omega_{\widetilde X}(E), \omega_X^\bullet)[-n] \cong R^j\pi_*\mathscr O_{\widetilde X}(-E) = 0$$ for $j > 0$.  Thus, $\pi_*\omega_{\widetilde X}(E)$ is reflexive.  We emphasize that this vanishing is much stronger than logarithmic extension and is only equivalent if $X$ is Cohen-Macaulay.  Theorem \ref{theorem logarithmic extension intro} follows more generally from the following lemma:

\begin{lemma}
If $X$ is a normal complex algebraic variety with a log-resolution of singularities $\pi:\widetilde X \to X$, then $\mathscr H^0\mathbb D_X(\underline \Omega_X^0) \cong \pi_*\omega_{\widetilde X}(E)$, where $\mathbb D_X$ is the Grothendieck duality functor.
\end{lemma}

For holomorphic extension, the problem is more subtle.  The naive approach is to consider the support of the sheaves $R^j\pi_*\Omega_{\widetilde X}^{\dim X - p}$ for $p = \mathrm{codim}_X(\Sigma)$, but this seems hopeless: for instance, $R^{1}\pi_*\Omega_{\widetilde X}^1$ is non-zero for ADE surface singularities --- these singularities satisfy holomorphic extension in all degrees, either classically or by \cite[Corollary 1.8]{kebekus2021extending}.  Instead, we use the brilliant approach of Kebekus-Schnell and the theory of Hodge modules.  

Let $X$ be a normal complex algebraic variety with Du Bois singularities.  There are objects $K_p, K_p' \in \mathrm{D}_{\mathrm{coh}}^b(\mathscr O_X)$ such that $K_p$ defines a sub-object of $\mathbf R\pi_*\Omega_{\widetilde X}^p$ for every $p$, and $$\pi_*\Omega_{\widetilde X}^p \cong \mathscr H^0K_p, \quad \pi_*\Omega_{\widetilde X}^p(\log E) \cong \mathscr H^0K_p'.$$  The objects $K_p$ are defined by the intersection cohomology complex and its data as a pure Hodge module; the objects $K_p'$ are determined by the data of a mixed Hodge module, which we refer to as the logarithmic mixed Hodge module.  Both are uniquely determined as extensions of the trivial Hodge module $\mathbb Q_{\mathrm{reg}}$. 

Assuming that $X$ has Du Bois singularities, Theorem \ref{theorem logarithmic extension intro} and \cite[\S 9]{kebekus2021extending} imply a family of support conditions $$\dim \mathrm{Supp}~ R^j\mathscr Hom_{\mathscr O_X}(K_p', \omega_X^\bullet) \le -(k + 2).$$ The key is to relate this support condition to $K_p$.  In particular, an inductive argument on $\dim \Sigma$ allows us to prove the following, which gives a stronger result than Theorem \ref{theorem holomorphic extension intro}:

\begin{theorem}
Let $X$ be a normal complex algebraic variety with at worst Du Bois singularities.  For $p = \mathrm{codim}_X(\Sigma) - 1$, we have $$\dim \mathrm{Supp}~\mathscr H^{j+n-p} K_{\dim X-p} \le -(j + 2),$$ where $$K_p : = \mathrm{gr}_{-p}^F\mathrm{DR}(IC_X)[p-n]$$ is the $p^{th}$-graded piece of the intersection cohomology Hodge module with its induced Hodge filtration. 
\end{theorem}

\subsection{Log Forms to Holomorphic Forms} It is known that Flenner's criterion for holomorphic extension is optimal for general singularities; in particular, there is a variety $X$ with non-Du Bois singularities for which holomorphic extension fails in degree $p = \mathrm{codim}_X(\Sigma)-1$.  We consider then a weakening of holomorphic extension via \ref{equation inclusion holomorphic log}: if a differential form on $X_{\mathrm{reg}}$ extend with at worst log-poles, does it already extend holomorphically?  Of course, (\ref{equation inclusion holomorphic log}) is surjective whenever holomorphic extension holds in degree $p$, and so this holds for any variety with $p < \mathrm{codim}_X(\Sigma) - 1$.  We consider again what happens for $p = \mathrm{codim}_X(\Sigma) - 1$.  This generalizes the discussion in \S\ref{subsection proof of 1.2}. 

\begin{theorem} \label{theorem log forms to holomorphic forms}
    Let $X$ be a normal variety with singular locus $\Sigma$.  The inclusion morphism (\ref{equation inclusion holomorphic log}) is an isomorphism for $p = \mathrm{codim}_X(\Sigma) - 1$.
\end{theorem}

\subsection{Holomorphic Extension for $m$-Du Bois Singularities} Du Bois singularities are of fundamental interest to algebraic geometers as a general class of singularities for which deformation theory and Hodge theory can be studied.  It is interesting to consider $m$-Du Bois singularities, the class of singularities which restricts the higher graded pieces of the Du Bois complex.

We say that a normal complex algebraic variety $X$ has \textit{weakly} (or \textit{pre}-) $m$-Du Bois singularities if the cohomology sheaves $\mathscr H^j\underline \Omega_X^p = 0$ for $j > 0$ and $0 \le p \le m$, where $\underline \Omega_X^p : = \mathrm{gr}^F_p\underline \Omega_X^\bullet[p]$.  Alternatively, $X$ has weakly $m$-Du Bois singularities if the natural map $$\mathscr H^0\underline \Omega_X^p \to \underline \Omega_X^p$$ is a quasi-isomorphism.  This notion has been studied in \cite{shen2023k} and is a generalization of the notion of \textit{$m$-Du Bois singularities} appearing in \cite{mustata2021bois}, \cite{jung2021higher}, and \cite{mustata2021hodge}.

An interesting question is then to what extent Theorem \ref{theorem holomorphic extension intro} can be improved for weakly $m$-Du Bois singularities.  For hypersurfaces, the answer is known: a variety $X$ with hypersurface singularities satisfies holomorphic extension in degree $\dim X$ if $X$ is 1-Du Bois, which means that $X$ is Du Bois and the natural map $$\Omega_X^1 \xrightarrow{\sim} \underline \Omega_X^1$$ is a quasi-isomorphism, where $\Omega_X^1$ is the sheaf of K\"ahler 1-forms.  This seems to be rather special for hypersurface (or more generally complete intersection) singularities, as we can write down a Cohen-Macaulay variety with weakly $m$-Du Bois singularities and $k > 0$ in any dimension for which Theorem \ref{theorem holomorphic extension intro} is still optimal (see Example \ref{example 1-Du Bois not Rational}).  The algebraic properties of the sheaf $\mathscr H^0 \underline \Omega_X^p$ appear to be independent of the vanishing --- or non-vanishing --- of the sheaves $\mathscr H^j\underline \Omega_X^p$.  Instead, we consider what happens when $\mathscr H^0\underline \Omega_X^p$ is already reflexive.  One result in this direction is the following:

\begin{theorem}
Let $X$ be a normal variety with weakly $m$-Du Bois singularities, and let $\pi:\widetilde X \to X$ be a resolution of singularities.  If $\mathscr H^0\underline \Omega_X^k$ is reflexive, then $\pi_*\Omega_{\widetilde X}^{k + 1}$ is reflexive.
\end{theorem}

\subsection{Higher Rational Singularities and Reflexivity of $\Omega_X^p$} Let $X$ be a normal complex algebraic variety.  Following \cite{mustata2021bois} and \cite{jung2021higher}, we say that $X$ has $m$-Du Bois singularities if the natural map $\Omega_X^p \to \underline \Omega_X^p$ is a quasi-isomorphism for each $p$, where $\Omega_X^p$ is the sheaf of K\"ahler $p$-forms.  As a generalization of Du Bois singularities, this is very natural, as it comes from a morphism of complexes $\Omega_X^\bullet \to \underline \Omega_X^\bullet$.  However, this definition is rather restrictive: in the case of hypersurface singularities, this forces $\Omega_X^p$ to be reflexive for every $0 \le p \le k$ and restricts the codimension of the singularities.

Instead, one can consider what happens when we restrict the \textit{duals} of the higher Du Bois complexes.  We say that has \textit{weakly $m$-rational singularities} if $$\mathscr H^0\mathbb D_X(\underline \Omega_X^{n-p}) \xrightarrow{\sim} \mathbb D_X(\underline \Omega_X^{n-p})$$ for each $0 \le p \le k$, where $\mathbb D_X$ is the Grothendieck duality functor.  Note that there is a quasi-isomorphism $\mathbb D_X(\underline \Omega_X^n) \cong \mathbf R\pi_*\mathscr O_{\widetilde X}$, whence $0$-rational is the same as rational singularities for normal complex varieties.  This is a generalization of the \textit{$m$-rational} definition established for complete intersections, which requires the natural morphism $\Omega_X^p \xrightarrow{\sim} \mathbb D_X(\underline \Omega_X^{n-p})$ to be an isomorphism for $0 \le p \le k$.  

The $m$-rational property is very strong, as it eventually implies $\Omega_X^1$ is \textit{maximal Cohen-Macaulay}.  The MCM property has been extensively studied in the literature, and the MCM property for the module of K\"ahler differentials has been looked at in the case of hypersurfaces and when $\mathrm{pd}~\Omega_{X,x}^p < \infty$.  We give one new result using the theory of Hodge modules.

\begin{proposition}
If $X$ is an $n$-fold Gorenstein variety with at worst quotient singularities, then $\Omega_X^p$ is reflexive for some $1 \le p \le n$ if and only if $X$ is smooth.
\end{proposition}

\subsection{Hodge Theory of Weakly $m$-rational Singularities} As we mentioned above, the Du Bois complex is a Hodge-theoretic object arising from Deligne's mixed Hodge theory.  If $X$ is a proper variety, then the Hodge filtration is the one induced by the $E_1$-spectral sequence $$E_1^{p,q} = \mathbb H^q(X, \underline \Omega_X^p) \Rightarrow H^{p+q}(X, \mathbb C).$$ If $X$ is weakly $m$-Du Bois, we get a decomposition $$H^k(X, \mathbb C) \cong \bigoplus_{p + q = k} H^q(X, \mathscr H^0\underline \Omega_X^p).$$ In general, restricting the Du Bois complex does not affect the weight filtration on the cohomology, as $\underline \Omega_X^\bullet$ does not usually admit the structure of a Hodge module.  To get pure Hodge modules, we consider again $m$-rational definition.

\begin{theorem}
If $X$ is a normal and proper complex algebraic variety with weakly $m$-rational singularities, then the Hodge filtration on $H^k(X, \mathbb Q)$ induces a pure Hodge structure.
\end{theorem}  

 \subsection{Acknowledgements} Part of this paper is contained in the author's PhD thesis.  I want to thank Benjamin Bakker and Nicolas Addington for conversations regarding this topic.  I want to thank Mihnea Popa, Wanchun Shen, and Duc Vo for comments on an early draft of this paper.  Finally, I thank Sung Gi Park, who independently proved Theorem \ref{theorem holomorphic extension intro} and Theorem \ref{theorem logarithmic extension intro} in \cite{park2023duBois}, for conversations on this topic. 

 \section{The Du Bois Complex}

 \subsection{Notation}\label{subsection notation} Throughout, we let $X$ be a normal complex algebraic variety of dimension $n$ with singular locus $\Sigma$. We will use $\pi:\tilde X \to X$ to denote a \textit{projective} resolution of singularities.

If $\pi:\tilde X \to X$ is a log-resolution of singularities with exceptional divisor $E$, we let $\Omega_{\tilde X}^p(\log E)$ be the sheaf of log $p$-forms, and we let $$\Omega_{\tilde X}^p(\log E)(-E) : = \Omega_{\tilde X}^p(\log E)\otimes \mathscr I_E.$$ For any local embedding $X|_V \hookrightarrow Y$ into a smooth complex variety, we define the sheaf of K\"ahler differentials $$\Omega_X^1|_V : = \Omega_Y^1/\langle df_1,df_2,...,df_m\rangle,$$ where $\Omega_Y^1$ is the sheaf of holomorphic 1-forms on $Y$ and $f_1,...,f_m$ are some defining equations for the open set $V$.  We let $\Omega_X^p : = \wedge^p\Omega_X^1$.

If $f:X' \to X$ is a morphism, we write $\mathbf Rf_*$ for the derived pushforward, and $\mathbf R\mathscr Hom$ for the derived hom.  We will use $H^k$ for the cohomology of a sheaf, $\mathscr H^k$ for the cohomology sheaf of a complex, and $\mathbb H^k$ for hypercohomology of a complex.  Finally, we let $\mathbb D_X(-): = \mathbf R\mathscr Hom_{\mathscr O_X}(-, \omega_X^\bullet)$ be the Grothendieck duality functor.
\subsection{The Du Bois Complex}

Let $X$ be a complex algebraic variety.  The Du Bois complex $(\underline \Omega_X^\bullet, F)$ is an object in the derived category of filtered complexes of constructible sheaves $\mathscr D^b_{\mathrm{filt}}(X)$, generalizing the holomorphic de Rham complex for algebraic varieties over $\mathbb C$.  We denote its graded pieces by $\underline \Omega_X^p : = \mathrm{gr}^F_p\underline \Omega_X^\bullet[p]$, which are defined in the bounded derived category of coherent sheaves.  Studied by Du Bois \cite{du1981complexe} from Deligne's construction of the mixed Hodge structure \cite{deligne1974theorie}, the Du Bois complex is constructed by simplicial or cubical hyperresolutions of $X$.  We will not need this construction in this paper but will only use its formal consequences.  The interested reader can consult \cite{peters2008mixed} for a good treatment of this construction.

\begin{theorem} \label{theorem Du Bois properties}
For $X$ a complex scheme of finite type and $\underline \Omega_X^{\bullet}$ its Du Bois complex, we have

\begin{enumerate}[label=\normalfont(\roman*)]
    \item \cite[\S 3.2]{du1981complexe} $\underline \Omega_X^{\bullet} \cong_{\mathrm{qis}} \mathbb C_X$.
    
    \item \cite[(3.2.1)]{du1981complexe}If $f:Y \to X$ is a proper morphism of finite type schemes, then there is a morphism $f^*: \underline \Omega_X^{\bullet} \to \mathbf Rf_*\underline \Omega_Y^{\bullet}$ in  $\mathscr D_{\mathrm{filt}}^b(X)$.
    
    \item\label{theorem Du Bois open restriction}\cite[3.10 Corollaire]{du1981complexe} If $U \subset X$ is an open subscheme then $\underline \Omega_X^{\bullet}|_U \cong_{\mathrm{qis}} \underline \Omega_U^{\bullet}$.

    \item \label{theorem property dim support Du Bois} \cite[V, 3.6]{guillen2006hyperresolutions} $\dim \mathrm{Supp}~\mathscr H^j\underline \Omega_X^p \le \dim X - j$ for $0 \le j \le \dim X$.
    
    \item \label{theorem Du Bois properties natural map kahler} \cite[\S 3.2]{du1981complexe} There is a natural morphism $\Omega_X^{\bullet} \to \underline \Omega_X^{\bullet}$, where $\Omega_X^{\bullet}$ is the complex of K\"ahler differentials.  Moreover, this morphism is a quasi-isomorphism if $X$ is smooth.
    
    \item\label{theorem Du Bois base point free} \cite[Proposition 2.6]{kovacsdu2011}, \cite[Lemma 3.2]{shen2023k} If $H \subset X$ is a general member of a basepoint free linear system, then there is an exact triangle $$\underline \Omega_H^{p-1} \otimes_{\mathscr O_H}^{\mathbb L} \mathscr O_H(-H) \to \underline \Omega_X^p\otimes_{\mathscr O_X}^{\mathbb L} \mathscr O_H \to \underline \Omega_H^p \xrightarrow{+1}.$$ In particular, $\underline \Omega_X^0\otimes \mathscr O_H \cong_{\mathrm{qis}} \underline \Omega_H^0$.
    
    \item \label{theorem right triangle resolution exceptional divisor} \cite[4.11 Proposition]{du1981complexe} There is an exact triangle $$\underline \Omega_X^p \to \underline \Omega_{\Sigma}^p\oplus \mathbf R\pi_*\Omega_{\tilde X}^p \to \mathbf R\pi_*\underline \Omega_E^p \xrightarrow{+1}$$ where $\pi:\tilde X \to X$ is a resolution of singularities, $\Sigma \subset X$ is the singular locus, and $E = \pi^{-1}(\Sigma)$. In particular, $\underline \Omega_X^n \cong_{\mathrm{qis}} \mathbf R\pi_*\omega_{\tilde X} \cong_{\mathrm{qis}} \pi_*\omega_{\tilde X}$. 
    
    \item \cite[4.5 Th\'eor\`eme]{du1981complexe} If $X$ is a proper variety, there is a spectral sequence $$E_1^{p,q} : = \mathbb H^q(X, \underline \Omega_X^p)\Rightarrow H^{p+q}(X, \mathbb C)$$ which degenerates at $E_1$ for every $p,q$.  Moreover, the filtration induced by this degeneration is equal to the Hodge filtration on the underlying mixed Hodge structure.

    \item\label{theorem right triangle log zero Du Bois} \cite[\S 3.C]{kovacsdu2011} If $\pi:\tilde X \to X$ is a log-resolution of singularities with exceptional divisor $E$, there is a right triangle $$\mathbf R\pi_*\Omega_{\tilde X}^p(\log E)(-E) \to \underline \Omega_X^p \to \underline \Omega_\Sigma^p \xrightarrow{+1}$$ which is independent of the choice of $\pi$.  
\end{enumerate}
\end{theorem}

\subsection{Du Bois, $m$-Du Bois, and $m$-rational Singularities} Let $X$ be a complex algebraic variety.  We review classes of singularities which are associated to the complexes $\underline \Omega_X^p$. 

\begin{definition}
Let $X$ be a complex algebraic variety.  

\begin{enumerate}[label=\normalfont(\roman*)]
    \item We say that $X$ has \textit{Du Bois singularities} if the natural map $\mathscr O_X \to \underline \Omega_X^0$ is a quasi-isomorphism.  
    \item We say that $X$ has \textit{$m$-Du Bois singularities} if the natural map $\Omega_X^p \to \underline \Omega_X^p$ is a quasi-isomorphism for each $0 \le p \le m$.

    \item We say that $X$ has \textit{$m$-rational singularities} if the natural map $\Omega_X^p \to \mathbb D_X(\underline \Omega_X^{n-p})$ is a quasi-isomorphism for each $0 \le p \le m$.

    \item We say that $X$ has \textit{weakly $m$-Du Bois singularities} if the natural map $\mathscr H^0\underline \Omega_X^p \to \underline \Omega_X^p$ is a quasi-isomorphism for each $0 \le p \le m$.  

    \item We say that $X$ has \textit{weakly $m$-rational singularities} if the natural map $\mathscr H^0\mathbb D_X(\underline \Omega_X^{n-p}) \to \mathbb D_X(\underline \Omega_X^{n-p})$ for each $0 \le p \le m$.
\end{enumerate}
\end{definition}

Here are some general properties of (weakly) $m$-Du Bois (resp. (weakly) $m$-rational singularities).

\begin{itemize}
    \item If $H$ is a general member of a basepoint linear system of a variety $X$ with (weakly) $m$-Du Bois singularities, then $H$ also has (weakly) $m$-Du Bois singularities (resp. (weakly) $m$-rational singularities) \cite[Theorem A, Corollary 3.3]{shen2023k}.

    \item Rational and log-canonical singularities are Du Bois, see \cite[Corollary 2.6]{kovacs1999rational}, \cite[5.4. Theorem]{saito2000mixed}, and \cite[Theorem 1.4]{kollar2010log}.

    \item If $X$ is a normal variety and $f:Y \to X$ is a finite dominant map from a variety $Y$ with rational and weakly $m$-Du Bois singularities, then $X$ also has rational \cite[Theorem 1]{kovacs2000characterization} and weakly $m$-Du Bois (resp. weakly $m$-rational) singularities \cite[Proposition 4.2]{shen2023k}.

    \item If $X$ has simple normal crossing singularities, then $X$ is weakly $m$-Du Bois since $$\underline \Omega_X^p \cong \Omega_X^p/\mathrm{tor}$$ for each $0 \le p \le n$.

    \item If $X$ has rational singularities, then $\mathscr H^0\underline \Omega_X^p \cong \Omega_X^{[p]}$ for each $p$, which follows from \cite{kebekus2021extending} and \cite{huber14}.  In particular, a complex algebraic variety $X$ with rational singularities is weakly $m$-Du Bois (resp. $m$-rational) if and only if $\Omega_X^{[p]} \to \underline \Omega_X^p$ (resp. $\Omega_X^{[p]} \to \mathbb D_X(\underline \Omega_X^{n-p})$) is a quasi-isomorphism for each $0 \le p \le k$.
\end{itemize}

There is a nice characterization due to Schwede.  For $X$ a reduced and separated complex scheme of finite type, there exists a (local) embedding $\iota:X \to Y$ of $X$ into a smooth scheme $Y$.  Let $\pi:\widetilde Y \to Y$ be an embedded resolution of $X$ which is an isomorphism outside of $X$, and let $\overline X = \pi^{-1}(X)_{\mathrm{red}}$ be the reduced preimage.  

\begin{proposition}\label{proposition schwede criterion} \cite[Theorem 4.6]{schwede2007simple} A complex algebraic variety $X$ has Du Bois singularities if and only if the natural map $\mathscr O_X \to \mathbf R\pi_*\mathscr O_{\overline X}$ is a quasi-isomorphism.   
\end{proposition}

As we mentioned, the Du Bois complex depends on the existence of hyperresolutions of singularities.  For lack of a reference, we remark that the conditions of Schwede's criterion hold in the analytic category, as embedded resolutions of singularities exist by \cite{bierstone1997canonical}.  

\begin{definition} \label{definition Du Bois analytic}
    Let $X$ be a complex analytic variety.  We say that $X$ has \textit{Du Bois singularities} if for (local) embedding $X \subset Y$ and any embedded resolution of singularities $\pi:\widetilde Y \to Y$ which is an isomorphism outside of $X$, the canonical morphism $$\mathscr O_X \to \mathbf R\pi_*\mathscr O_{\overline X}$$ is a quasi-isomorphism, where $\overline X : = \pi^{-1}(X)$ is the reduced preimage.
\end{definition}

Unfortunately, Schwede's criterion fails for the higher graded pieces, as the proof depends on the fact that simple normal crossing singularities are (0-) Du Bois.  We note that a hyperresolution-free description of the complexes $\underline \Omega_X^p$ has been given in \cite{hampton2023hyperresolution}.

\section{Hodge Modules and Differentials on the Resolution}

\subsection{Mixed Hodge Modules} (Mixed) Hodge modules are generalizations of variations of (mixed) Hodge structures in the presence of singularities.  We review some definitions concerning mixed Hodge modules following \cite{kebekus2021extending}.  The interested reader may also refer to \cite{schnell2014overview} for more details.

\subsubsection{Pure and Mixed Hodge Modules}\label{subsubsection pure and mixed Hodge module} Let $Y$ be a smooth complex manifold of dimension $d$ (for example, let $Y$ be a local embedding of a complex algebraic variety $X$ of codimension $c = d-n$).
    A \textit{pure Hodge module} on $Y$ is an object $M = (\mathcal M, F^\bullet, \mathrm{rat}~M)$ consisting of:

    \begin{enumerate}[label=\normalfont(\roman*)]
        \item A regular holonomic left $\mathscr D_Y$-module $\mathcal M$, where $\mathscr D_Y$ is the sheaf of differential operators on $Y$;

        \item \label{item good filtration} An increasing \textit{good} filtration $F_\bullet \mathcal M$ of coherent $\mathscr O_Y$-modules, called the Hodge filtration, which is compatible with the $\mathscr D_Y$-module structure: $$F_p\mathcal M\cdot F_q\mathscr D_Y \subset F_{p + q}\mathcal M,$$ and $\mathrm{gr}^F_\bullet \mathcal M$ is coherent over $\mathrm{gr}^F_\bullet \mathscr D_Y$.

        \item\label{item perverse structure} A perverse sheaf $\mathrm{rat}~M$ of $\mathbb Q$-vector spaces satisfying $$\mathrm{rat}~M\otimes \mathbb C \cong \mathrm{DR}(\mathcal M),$$ where $\mathrm{DR}$ is the de Rham complex $$\mathrm{DR}(\mathcal M): = \big[\mathcal M \to \Omega_Y^1\otimes_{\mathscr O_Y} \mathcal M \to ... \to \Omega_Y^{d}\otimes_{\mathscr O_Y} \mathcal M\big][d]$$ associated to the $\mathscr D_Y$-module $\mathcal M$.  In particular, we have the support condition \begin{equation}\label{equation perverse support}
        \dim \mathrm{Supp}~\mathscr H^j\mathrm{DR}(\mathcal M) \le -j.
    \end{equation}
    \end{enumerate}

    \begin{example}
     If $Y$ is a smooth complex manifold of dimension $d$, the locally constant sheaf $\mathbb Q_Y$ admits the structure of a variation of pure Hodge structures.  It therefore inherits a Hodge module structure.  The underlying left $\mathscr D_Y$-module is simply $\mathscr O_Y$, considered as a subsheaf of $\mathscr{E}nd_{\mathbb C_Y}(\mathscr O_Y)$.  The Hodge filtration is the trivial filtration $F_p\mathscr O_Y = \mathscr O_Y$ for $p \ge 0$ (and zero otherwise).  The perverse structure gives the identification $$\mathbb C_Y[d] \cong_\mathrm{qis} \mathrm{DR}(\mathscr O_Y) = \big[\mathscr O_Y \xrightarrow{d} \Omega_Y^1 \xrightarrow{d} ... \xrightarrow{d} \Omega_Y^{d}\big][d].$$
    \end{example}

    One can generalize the previous construction to any variation of pure Hodge structure.  Conversely, pure Hodge modules satisfy \textit{decomposition by strict support} \cite[\S 5]{saito1988modules}: a Hodge module is completely determined by a collection of variations of pure Hodge structures $M_{Y_i}$ supported on a stratification $\left\{Y_i\right\}$ of $Y$, of weights $w - \dim Y_i$, for some $w$.  We refer $w$ as the weight of the Hodge module.  Because of this, we can define the category $\mathrm{HM}(Y,w)$ of \textit{polarized Hodge modules} of weight $w$ by inducing a polarization from the data $M_{Y_i}$ of polarized variations of pure Hodge structures. \newline

    A \textit{mixed Hodge module} on $Y$ is an object $M = (\mathcal M, F^\bullet, W_\bullet, \mathrm{rat}~M)$ consisting of a $\mathscr D_Y$-module structure $\mathcal M$, a perverse structure $\mathrm{rat}~M$, a decreasing good filtration $F^\bullet$ on $\mathcal M$, and an increasing \textit{weight} filtration $W_\bullet$ on these structures such that the graded pieces $$\mathrm{gr}_W^kM = (\mathrm{gr}_W^k\mathcal M, \mathrm{rat}~\mathrm{gr}_W^k M, F^\bullet)$$ are pure Hodge modules.  We call $M$ a \textit{graded-polarizable} mixed Hodge module if further $\mathrm{gr}_W^k M$ is polarizable for each $k$.  We refer to the category of graded-polarizable mixed Hodge modules as $\mathrm{MHM}(Y)$.

    \subsubsection{The Dual Mixed Hodge Module} Let $Y$ be a complex manifold of dimension $d$, and fix a mixed Hodge module $M = (\mathcal M, F^\bullet,  W_\bullet, \mathrm{rat}~M) \in \mathrm{MHM}(Y)$. 
 There is a mixed Hodge module $\mathbf D_Y(M)$, known as the \textit{dual Hodge module}, satisfying \begin{equation}\label{equation dual mixed Hodge module}\mathbf D_Y(\mathrm{gr}_W^k M) = \mathrm{gr}_{-k}^W\mathbf D_Y(M)\end{equation} for every $k$.  The underlying perverse sheaf $\mathrm{rat}~\mathbf D_Y(M)$ is the Verdier dual of $\mathrm{rat}~M$, and the underlying $\mathscr D_Y$-module is the holonomic dual $$\mathbf D_Y(\mathcal M) = R^{d}\mathrm{Hom}_{\mathscr D_Y}(\omega_Y\otimes_{\mathscr O_Y}\mathcal M, \mathscr D_Y),$$ which is compatible with the filtration $F^\bullet$ of item \ref{item good filtration} of \S \ref{subsubsection pure and mixed Hodge module}, by \cite[Lemma 5.1.13]{saito1988modules}.  In particular, if $M$ is a pure Hodge module of weight $w$, then $\mathbf D_Y(M)$ is a pure Hodge module of weight $-w$.  In this case, $\mathbf D_Y(M) \cong M(w)$, where $M(w)$ is the \textit{Tate-twist} of $M$ in degree $w$, obtained by twisting the perverse structure and Hodge filtration in the usual way.

    The compatibility of the Hodge filtration with $\mathbf D_Y$ is due to Saito \cite[2.4.3]{saito1988modules}.  Specifically, there is an isomorphism
    \begin{equation}\label{equation duality mixed} \mathbf R\mathscr Hom_{\mathscr O_Y}(\mathrm{gr}_p^F\mathrm{DR}(\mathcal M), \omega_Y^\bullet) \cong \mathrm{gr}_{-p}^F\mathrm{DR}(\mathbf D_Y(\mathcal M)),   
    \end{equation}
    where again $\omega_Y^\bullet = \omega_Y[d]$.  If $M \in \mathrm{HM}(Y,w)$, the isomorphism $\mathbf D_Y(M) \cong M(w)$ reduces (\ref{equation duality mixed}) to 
    \begin{equation}\label{eqution duality pure}
     \mathbf R\mathscr Hom_{\mathscr O_Y}(\mathrm{gr}_p^F\mathrm{DR}(\mathcal M), \omega_Y^\bullet) \cong \mathrm{gr}_{-p-w}^F\mathrm{DR}(\mathcal M).   
    \end{equation}

   \subsubsection{The Restricted Hodge Module} Now let $Y = \mathbb C^{d}$ and $M = (\mathcal M, F^\bullet, W_\bullet, \mathrm{rat}~M) \in \mathrm{MHM}(Y)$.  For a generic hyperplane section $H$ of $Y$, $H$ intersects any Whitney strata adapted to the perverse sheaf $\mathrm{DR}(\mathcal M)$ transversely.  Therefore $H$ defines a \textit{non-characteristic hypersurface} with respect to the left $\mathscr D_Y$-module $\mathcal M$, see \cite[Definition 4.15]{kebekus2021extending} and \cite[\S 9]{schnell16saito}.  Given such a hyperplane $H \subset Y$, we can construct a mixed Hodge module $M_H \in \mathrm{MHM}(H)$ by \cite[Lemma 2.25]{saito1990mixed}.  If $\iota_H: H \hookrightarrow Y$ is the inclusion, the underlying $\mathscr D_H$-module is $$\mathcal M_H = \mathscr O_H\otimes_{\iota_H^{-1}\mathscr O_Y} \iota_H^{-1} \mathcal M$$ with filtration $F^\bullet \mathcal M_H = \mathscr O_H\otimes_{\iota_H^{-1}\mathscr O_Y}\iota_H^{-1}F^\bullet \mathcal M$, see \cite[Lemma 9.5]{schnell16saito}.  The de Rham complex of $\mathcal M_H$ is $$\mathrm{DR}(\mathcal M_H) = \iota_H^{-1}\mathrm{DR}(\mathcal M)[-1].$$  This give the data of a mixed Hodge module $M_H$ on $H$.

   The restricted Hodge module will be important for many inductive arguments, as we have the following generalization of the conormal bundle sequence \cite[(13.3)]{schnell2014overview}: \begin{equation}\label{equation conormal cutdown} 0 \to N_{H|Y}^*\otimes_{\mathscr O_H} \mathrm{gr}_{p+1}^F\mathrm{DR}(\mathcal M_H) \to \mathscr O_H \otimes_{\mathscr O_Y}\mathrm{gr}_p^F\mathrm{DR}(\mathcal M) \to \mathrm{gr}_p^F\mathrm{DR}(\mathcal M_H)[1] \to 0, 
   \end{equation} where $N_{H|Y}^*$ is the conormal bundle of the inclusion $\iota_H$.

\subsection{The Intersection Hodge Module} Returning now to singularities, let $X$ be a reduced and (for convenience) irreducible complex analytic variety of dimension $n$, and let $\iota: X \hookrightarrow Y : = \mathbb C^{n+c}$ be a (local) closed embedding into the smooth open ball $Y$ of codimension $c$.  We consider an object $\mathrm{IC}_X \in \mathrm{HM}(Y, n)$, called the \textit{intersection Hodge module}, whose support is exactly $X$\footnote{More generally, we can consider the category of pure/mixed Hodge modules on $X$ by passing to local embeddings into smooth varieties: see \cite[\S 2]{saito1990mixed}}.  By \cite[Thm. 3.21]{saito1990mixed}, the category $\mathrm{PVHS}_{\mathrm{gen}}(X, w)$, which is the direct limit of polarized variations of Hodge structures with quasi-unipotent local monodromies over Zariski open subsets $U \subset X \subset Y$, is equivalent to the subcategory of $\mathrm{HM}(Y, w)$ of pure Hodge modules with \textit{strict support} $X$.  

\begin{definition} \label{definition intersection hodge module}
Let $X$ be a complex analytic variety of dimension $n$ and $\iota:X \hookrightarrow Y$ a smooth embedding into a ball $Y = \mathbb C^{n+c}$ of codimension $c$.  The \textit{intersection Hodge module} $\mathrm{IC}_X = (\mathcal{IC}_X, F, IC_X) \in \mathrm{HM}(Y,n)$ is the unique Hodge module with support $X$ determined by the variation of Hodge structures $\mathbb Q_U[n]$, where $U = X_{\mathrm{reg}}$ is the regular locus of $X$.   
\end{definition}

We remark that the underlying perverse sheaf $IC_X$ computes the intersection cohomology $IH^\bullet(X, \mathbb Q) : = \mathbb H^{\bullet-n}(X, IC_X)$ of $X$.  With this in mind, we get a Hodge-theoretic interpretation of the decomposition theorem. 

Let $\pi:\widetilde X \to X$ be a resolution of singularities, and consider the induced morphism $f = \iota\circ \pi:\widetilde X \to Y$.  By Saito's direct image theorem \cite[\S 5.3]{saito1988modules}, there are pure Hodge modules $M_l = (\mathcal M_l, F^\bullet, \mathrm{rat}~M_l)$ supported on the singularities of $X$ and a decomposition\begin{equation}\label{equation saito decomposition} \mathbf Rf_*\Omega_{\widetilde X}^p[n-p] \cong_{\mathrm{qis}} \mathrm{gr}_{-p}^F\mathrm{DR}(\mathcal{IC}_X)\oplus \bigoplus_{l \in \mathbb Z} \mathrm{gr}_{-p}^F\mathrm{DR}(\mathcal M_l)[-l],   
\end{equation} see \cite[(8.0.3)]{kebekus2021extending}.  As the Hodge modules $M_l$ are supported on the singularities, they are torsion; this gives \begin{equation} \label{equation holomorphic forms intersection complex} f_*\Omega_{\widetilde X}^p \cong \mathscr H^{-(n-p)}\mathrm{gr}_{-p}^F\mathrm{DR}(\mathcal{IC}_X)  
\end{equation} by \cite[Proposition 8.1]{kebekus2021extending}.

\subsection{The Logarithmic Hodge Module}

Let $X$ be a reduced and irreducible complex analytic variety and $\iota:X \hookrightarrow Y$ a (local) closed embedding into a complex ball $Y = \mathbb C^{n+c}$ of codimension $c$.  If $\pi:\widetilde X \to X$ is a log-resolution of singularities with log-exceptional divisor $E$, recall from \S\ref{subsection notation} that we assume an isomorphism $U : = X_{\mathrm{reg}} \cong \widetilde X\setminus E$.  The sheaf $\mathbb Q_{\widetilde X\setminus E}[n]$ defines a variation of pure Hodge structures; as a perverse sheaf, we can consider the pushforward $\mathbf Rj_*\mathbb Q_{\widetilde X\setminus E}$, where $j:\tilde X \setminus E \hookrightarrow \widetilde X$ is the inclusion. By \cite[Thm. 3.27]{saito1990mixed}, this extends to a graded-polarizble mixed Hodge module on the smooth manifold $\widetilde X$.  Specifically, the underlying $\mathscr D_{\widetilde X}$-module is the sheaf $\mathscr O_{\widetilde X}(*E)$ of meromorphic functions which are regular outside $E$.  The filtered pieces $F_p \mathscr D_{\widetilde X}$ act naturally on $\mathscr O_{\widetilde X}(*E)$, inducing a good filtration on $\mathscr O_{\widetilde X}(*E)$.  The de Rham complex is simply $$\mathrm{DR}(\mathscr O_{\widetilde X}(*E)) = \big[ \mathscr O_{\widetilde X}(*E) \xrightarrow{d} \Omega_{\widetilde X}^1(*E) \xrightarrow{d} ... \xrightarrow{d} \Omega_{\widetilde X}^n(*E)\big][n].$$ By \cite[Proposition 3.11]{saito1990mixed} (or classically), the inclusion $\Omega_{\widetilde X}^\bullet(\log E)[n]\hookrightarrow \mathrm{DR}(\mathscr O_{\widetilde X}(*E))$ is a filtered quasi-isomorphism: $$\Omega_{\widetilde X}^p(\log E)[n-p]\cong_{\mathrm{qis}} \mathrm{gr}_{-p}^F\mathrm{DR}(\mathscr O_{\widetilde X}(*E)).$$

Let $f = \iota\circ \pi:\widetilde X \to Y$. Saito's direct image theorem says there is a family of mixed Hodge modules $\left\{N_l\right\}_{l \in \mathbb Z}$ whose rational perverse structure $\mathrm{rat}~N_l$ is simply the $l^{th}$ perverse cohomology sheaf of $\mathbf Rf_*(j_*\mathbb Q_{\widetilde X\setminus E}[n]) \cong \mathbf Rj_*\mathbb Q_U[n]$.  As such, the support of $N_0$ is $X$, the support of $N_l$ for $l \ne 0$ is contained in $X_{\mathrm{sing}}$, and $N_0$ has no non-trivial sub-objects supported on $X_{\mathrm{sing}}$ \cite[Lemma 9.3]{kebekus2021extending}.  

\begin{definition}
Let $X$ be a reduced and irreducible complex analytic variety of dimension $n$, let $\iota: X \to Y$ be a (local) closed embedding into $Y = \mathbb C^{n+c}$, let $\pi:\widetilde X \to X$ be a log-resolution of singularities with log-exceptional divisor $E$, let $j:\widetilde X \setminus E \hookrightarrow \widetilde X$ be the inclusion, and let $f = \iota\circ \pi:\widetilde X \to Y$. 

The \textit{logarithmic mixed Hodge module} $N_0 \in \mathrm{MHM}(Y)$ is the unique mixed Hodge module supported on $X$ obtained from the push-forward $f_*(j_*\mathbb Q_{\widetilde X\setminus E})$ of perverse sheaves.  We refer to the underlying $\mathscr D_Y$-module as $\mathcal N_0$.
\end{definition}

By \cite[Proposition 9.5]{kebekus2021extending}, we get the following important relationship between logarithmic forms and the logarithmic Hodge module $N_0 = (\mathcal N_0, F^\bullet, W_\bullet, \mathrm{rat}~N_0)$:\begin{equation} \label{equation logarithmic forms logarithmic hodge module} f_*\Omega_{\widetilde X}^p(\log E) \cong \mathscr H^{-(n-p)} \mathrm{gr}_{-p}^F\mathrm{DR}(\mathcal N_0).   
\end{equation}

\subsection{An Extension Criterion for Differentials on a Resolution of Singularities} Let $X$ be a normal complex analytic variety.  By (\ref{equation holomorphic forms intersection complex}) and (\ref{equation logarithmic forms logarithmic hodge module}) the sheaves $\pi_*\Omega_{\widetilde X}^p$ and $\pi_*\Omega_{\widetilde X}^p(\log E)$ associated to a (log-)resolution of singularities can be recovered from the intersection Hodge module $\mathrm{IC}_X$ and the logarithmic Hodge module $N_0$, respectively.  By Proposition \ref{proposition kebekus criterion complexes}, the following gives a criterion to when the inclusions (\ref{equation holomorphic extension}) and (\ref{equation logarithmic extension}) are isomorphisms:

\begin{proposition} \label{proposition extension criterion Hodge module}
Let $X$ be a normal complex variety of dimension $n$, and let $X \subset Y$ be a (local) closed embedding into a smooth complex manifold $Y$. 

\begin{enumerate}[label=\normalfont(\roman*)]
    \item Holomorphic extension holds in degree $p \ge 0$ if \begin{equation} \label{equation holomorphic support condition}
        \dim \mathrm{Supp}~\mathscr H^{j + n-p}\mathrm{gr}_{-(n-p)}^F\mathrm{DR}(\mathcal{IC}_X) \le -(j+2)
    \end{equation} for every $j \ge 0$, where $\mathcal{IC}_X$ is the $\mathscr D_Y$-module underlying the intersection Hodge module $\mathrm{IC}_X$.

    \item Logarithmic extension holds in degree $p \ge 0$ if \begin{equation} \label{equation logarithmic support condition}
       \dim \mathrm{Supp}~\mathscr H^j\mathrm{gr}_{p}^F\mathrm{DR}(\mathbf D_Y(\mathcal N_0)) \le n-j-p-2 
    \end{equation} for every $j \ge 0$, where $\mathcal N_0$ is the $\mathscr D_Y$-module underlying the logarithmic Hodge module $N_0$.
\end{enumerate}
\end{proposition}

\begin{proof}
Rewrite (\ref{equation holomorphic forms intersection complex}) and (\ref{equation logarithmic forms logarithmic hodge module}) as $$f_*\Omega_{\widetilde X}^p \cong \mathscr H^0\mathrm{gr}_{-p}^F\mathrm{DR}(\mathcal{IC}_X)[p-n], \quad f_*\Omega_{\widetilde X}^p(\log E) \cong \mathscr H^0\mathrm{gr}_{-p}^F\mathrm{DR}(\mathcal N_0)[p-n].$$ Proposition \ref{proposition kebekus criterion complexes} says then that these global sections extend uniquely in codimension 2 if $$\dim \mathrm{Supp}~\mathscr H^{j+n-p}\mathbf R\mathscr Hom_{\mathscr O_Y}(\mathrm{gr}_{-p}^F\mathrm{DR}(\mathcal{IC}_X)[p-n], \omega_Y^\bullet) \le (-j+2)$$ and $$\dim \mathrm{Supp}~\mathscr H^j\mathbf R\mathscr Hom_{\mathscr O_Y}(\mathrm{gr}_{-p}^F\mathrm{DR}(\mathcal N_0)[p-n], \omega_Y^\bullet) \le -(j+2)$$ hold, respectively.  The first inequality is equivalent to (\ref{equation holomorphic support condition}) by duality for pure Hodge modules (\ref{eqution duality pure}), since $\mathrm{IC}_X$ has weight $n$, and the second inequality is equivalent to (\ref{equation logarithmic support condition}) by duality for mixed Hodge modules (\ref{equation duality mixed}).
\end{proof}

\section{Logarithmic Extension for Du Bois Singularities}

\subsection{Work of Kov\'acs-Schwede-Smith} Let $X$ be a normal variety of dimension $n$ with singular locus $\Sigma$.  It is well-known that holomorphic extension fails in degrees $p \ge \mathrm{codim}_X(\Sigma)$ for varieties with Du Bois singularities: the affine cone of a smooth and projective Calabi-Yau variety will be strictly log-canonical and will fail holomorphic extension in degree $n$.  On the other hand, logarithmic extension is known to hold in all degrees if we further assume $X$ is Cohen-Macaulay by \cite[Theorem 1.1]{kovacs2010canonical} and \cite[Theorem 1.5]{kebekus2021extending}.  The key input is to use Proposition \ref{proposition schwede criterion} to identify the sheaf $\pi_*\omega_{\widetilde X}(E)$ of logarithmic $n$-forms coming from a log-resolution $\pi:\widetilde X \to X$ with $\mathscr H^0\mathbb D(\underline \Omega_X^0)$.  We extend this proof with minor adjustments to the non-CM case. 

\subsection{Proof of Theorem \ref{theorem logarithmic extension intro}} 
\begin{lemma} \label{lemma KSS variant}
 Let $X$ be a normal complex algebraic variety of dimension $n$.  For any log-resolution of singularities $\pi:\widetilde X \to X$ with log-exceptional divisor $E$, then the natural map $$\mathscr H^0\mathbb D_X(\underline \Omega_X^0) \to\pi_*\omega_{\widetilde X}(E)$$ is an isomorphism, where $\mathbb D_X$ is the Grothendieck duality functor on $X$. 
\end{lemma}

\begin{proof}
Let $X \subset Y$ be a (locally) closed embedding into a smooth complex manifold $Y$, and consider the singular locus $\Sigma \subset X \subset Y$ under this embedding.  There is a distinguished triangle $$\mathbf R\mathscr Hom_{\mathscr O_X}(\underline \Omega_\Sigma^0, \omega_X^\bullet) \to \mathbf R\mathscr Hom_{\mathscr O_X}(\underline \Omega_X^0, \omega_X^\bullet) \to \mathbf R\pi_*\omega_{\widetilde X}(E)[n] \xrightarrow{+1}$$ obtained by dualizing the triangle of Theorem \ref{theorem Du Bois properties}\ref{theorem right triangle log zero Du Bois} for $p = 0$.  By \cite[Corollary 3.7]{kovacs2010canonical}, $R^j\mathscr Hom_{\mathscr O_X}(\underline \Omega_\Sigma^0, \omega_X^\bullet) = 0$ for $j < -\dim \Sigma$.  Specifically, there is a spectral sequence $$E_2^{p,q} = R^p\mathscr Hom_{\mathscr O_Y}(\mathscr H^{-q}\underline \Omega_\Sigma^0, \omega_Y^\bullet) \Rightarrow R^{p+q}\mathscr Hom_{\mathscr O_Y}(\underline \Omega_\Sigma^0, \omega_Y^\bullet).$$  Note that $\mathrm{Supp}~E_2^{p,q} \ne 0$ if $j = -\dim \Sigma,...,0$, whence $\dim \mathrm{Supp}~E_2^{p,q} = 0$ if $p < \dim \Sigma - q$ by Theorem \ref{theorem Du Bois properties}\ref{theorem property dim support Du Bois}.  Thus $E_\infty^{p,q} = 0$ in this range, which gives the desired vanishing on $X$.
\end{proof}

\noindent \textit{Proof of Theorem \ref{theorem logarithmic extension intro}}.  By assumption, $\underline \Omega_X^0$ has no higher cohomology, and $\dim \Sigma \le n-2$ since $X$ is normal.  Therefore logarithmic extension holds in degree $n$ by Lemma \ref{lemma KSS variant} and Proposition \ref{proposition kebekus criterion complexes}.  By \cite[Theorem 1.5]{kebekus2021extending}, logarithmic extension holds in all degrees $0 \le p \le n$. \qed 

\section{Holomorphic Extension for Du Bois Singularities}

\subsection{Consequences of Logarithmic Extension} Let $X$ be a normal variety of dimension $n$ and $\pi:\widetilde X \to X$ a log-resolution of singularities.  For dimension reasons and (\ref{equation logarithmic forms logarithmic hodge module}), there is a quasi-isomorphism $$\pi_*\omega_{\widetilde X}(E) \cong \mathrm{gr}_{-n}^F\mathrm{DR}(\mathcal N_0).$$  If $\pi_*\omega_{\widetilde X}(E)$ is reflexive, we get a family of inequalities \begin{equation*}
    \dim \mathrm{Supp}~\mathscr H^j\mathrm{gr}_p^F\mathrm{DR}(\mathbf D_Y(\mathcal N_0)(-n)) \le -(j + p +2)
\end{equation*} for $p + j \ge -n+1$, see \cite[Proposition 9.10]{kebekus2021extending}.  This leads to an important vanishing as a consequence of Theorem \ref{theorem logarithmic extension intro}.

\begin{corollary} \label{corollary vanishing Du Bois singularities log hodge module}
Let $X$ be a normal complex analytic variety of dimension $n$ with at worst Du Bois singularities.  Then $$\mathscr H^0\mathrm{gr}_{-1}^F\mathrm{DR}(\mathbf D_Y(\mathcal N_0)(-n)) = 0,$$ where $X \subset Y$ is a (locally) closed embedding into a smooth manifold $Y$ and $\mathcal N_0$ is the $\mathscr D_Y$-module underlying the logarithmic mixed Hodge module.  
\end{corollary}

This is the main component to beginning the induction for the holomorphic extension property for Du Bois singularities, using the relationship between the weight filtration of the logarithmic Hodge module and the intersection Hodge module. 

\begin{proposition} \label{proposition vanishing intersection hodge module Du Bois}
Let $X$ be a normal variety of dimension $n$ with at worst Du Bois singularities, and let $\mathrm{IC}_X = (\mathcal IC_X, F^\bullet, IC_X)$ be the intersection Hodge module.  Then $$\mathscr H^0\mathrm{gr}_{-1}^F\mathrm{DR}(\mathcal{IC}_X) = 0.$$
\end{proposition}

\begin{proof}
Let $X \subset Y$ be a (locally) closed embedding into a smooth complex manifold $Y$.  Let $N_0 = (\mathcal N_0, F^\bullet, W_\bullet, \mathrm{rat}~N_0)$ be the logarithmic mixed Hodge module.  By construction, $\mathrm{gr}_{-n}\mathbf D_Y(\mathcal N_0) = \mathbf D_Y(W_n\mathcal N_0)$.  Since $\mathrm{gr}_{-1}^F\mathrm{DR}(-)$ is an exact functor, there is an exact sequence \begin{equation*}
\begin{split}
    \mathscr H^0\mathrm{gr}_{-1}^F\mathrm{DR}(\mathbf D_Y(\mathcal N_0)(-n)) &\to \mathscr H^0\mathrm{gr}_{-1}^F(\mathrm{gr}_n^W\mathbf D_Y(\mathcal N_0)(-n))\\ & \to \mathscr H^1\mathrm{gr}_{-1}^F\mathrm{DR}(W_{n-1}\mathbb D(\mathcal N_0)(-n)).
    \end{split}
\end{equation*} The last term must be zero for degree reasons, and so by Corollary \ref{corollary vanishing Du Bois singularities log hodge module} we get the additonal vanishing $\mathscr H^0\mathrm{gr}_{-1}^F\mathrm{DR}(\mathrm{gr}_n^W\mathbf D_Y(\mathcal N_0)(-n)) = 0$.  By (\ref{equation duality mixed}), we have an isomorphism $\mathrm{gr}_n^W\mathbf D_Y(\mathcal N_0)(-n) \cong \mathrm{gr}_n^W\mathcal N_0$.  The claim follows from the isomorphism $$\mathrm{gr}_{-1}^F\mathrm{DR}(\mathrm{gr}_n^W\mathcal N_0) \cong_{\mathrm{qis}} \mathrm{gr}_{-1}^F\mathrm{DR}(\mathcal{IC}_X),$$ see \cite[Proposition 9.8]{kebekus2021extending}.  
\end{proof}

\subsection{Proof of Theorem \ref{theorem holomorphic extension intro}}\label{subsection proof of 1.2} Let $X$ be a normal variety of dimension $n$ with Du Bois singularities.  Recall from Proposition \ref{proposition extension criterion Hodge module} that holomorphic extension holds in degree $p$ if \begin{equation}\label{equation dim support proof 1.2}\mathrm{dim}~\mathrm{Supp}~\mathscr H^{j+n-p}\mathrm{gr}_{-(n-p)}^F\mathrm{DR}(\mathcal{IC}_X) \le -(j+2)\end{equation}  Note that this condition vacuously holds for isolated singularities except possibly when $(p,j) = (n-1,0)$, which is covered by Proposition \ref{proposition vanishing intersection hodge module Du Bois}.  Therefore, Theorem \ref{theorem holomorphic extension intro} holds for isolated singularities (compare to \S\ref{subsection intro proof 1.2}). \newline

\noindent \textit{Proof of Theorem \ref{theorem holomorphic extension intro}}. We have demonstrated (\ref{equation dim support proof 1.2}) for normal surfaces with Du Bois singularities. We proceed then by induction on the pair $(n, \dim \Sigma)$, ordered lexicographically.  If $X \subset Y$ a closed embedding into a smooth ball $Y = \mathbb C^{n+c}$, let $H \subset Y$ be a general hyperplane section and consider the exact sequence of complexes $$0 \to N_{H|Y}^*\otimes_{\mathscr O_H} \mathrm{gr}_{p+1}^F\mathrm{DR}(\mathcal{IC}_X|_H) \to \mathscr O_H \otimes_{\mathscr O_Y}\mathrm{gr}_p^F\mathrm{DR}(\mathcal{IC}_X) \to \mathrm{gr}_p^F\mathrm{DR}(\mathcal{IC}_X|_H)[1] \to 0$$ of (\ref{equation conormal cutdown}).  Since $H$ is generic, the restricted Hodge module $\mathrm{IC}_X|_H = \mathrm{IC}_{X\cap H}$ since we can assume that $(X\cap H)_{\mathrm{reg}} = X_{\mathrm{reg}}\cap H$ (see Definition \ref{definition intersection hodge module}).  Since $H$ is Du Bois by \ref{theorem Du Bois properties}\ref{theorem Du Bois base point free}, then (\ref{equation dim support proof 1.2}) and induction imply $$\dim \mathrm{Supp}~\mathscr O_H\otimes \mathscr H^{j + (n-1) - p}\mathrm{gr}_{-(n-p)}^F\mathrm{DR}(\mathcal{IC}_X) \le -(j+2)$$ for $p = \mathrm{codim}_X(\Sigma)$.  Therefore \begin{equation}\label{equation minimum support condition}
    \dim \mathrm{Supp}~\mathscr H^{j + n-p}\mathrm{gr}_{-(n-p)}^F\mathrm{DR}(\mathcal{IC}_X) \le -(j+2),
\end{equation} 
except possibly in the case $(p,j) = (1,0)$.  But this is exactly Proposition \ref{proposition vanishing intersection hodge module Du Bois}. \qed  

\subsection{Log Forms to Holomorphic Forms} Suppose now that $X$ is an arbitrary normal complex variety with possibly non-Du Bois singularities, and let $\Sigma$ be the singular locus.  It is known that holomorphic extension can fail for $p \ge \mathrm{codim}_X(\Sigma) - 1$.  We consider then the inclusion (\ref{equation inclusion holomorphic log}) and give a weakening of Theorem \ref{theorem holomorphic extension intro} to arbitrary singularities.  We remark that this theorem is optimal even for Du Bois singularities.  \newline

\noindent \textit{Proof of Theorem \ref{theorem log forms to holomorphic forms}}. Let $\mathrm{IC}_X$ and $N_0$ be the intersection and logarithmic Hodge modules of the normal variety $X$.  We note there is a canonical morphism of (mixed) Hodge modules $\mathcal{IC}_X \to N_0$ obtained by the isomorphism $W_n N_0 \cong \mathrm{IC}_X$ \cite[Proof of (9.8.1)]{kebekus2021extending}.  The problem is local, so let $X \subset Y$ be a closed embedding of $X$ into a ball $Y \cong \mathbb C^{n + c}$, and let $H \subset Y$ be a generic hyperplane section.  For each $p$, there is a commutative diagram of exact sequences

\[ \begin{tikzcd}
    0 \arrow{d} &0 \arrow{d} \\
 N_{H|Y}^* \otimes_{\mathscr O_H} \mathscr H^{-(n-p)}\mathrm{gr}_{-p+1}^F\mathrm{DR}(\mathcal{IC}_{X\cap H}) \arrow{r} \arrow{d}   & N_{H|Y}^*\otimes_{\mathscr O_H} \mathscr H^{-(n-p)}\mathrm{gr}_{-p+1}^F\mathrm{DR}(\mathcal{N}_0|_H) \arrow{d} \\
  \mathscr O_H \otimes_{\mathscr O_Y} \mathscr H^{-(n-p)} \mathrm{gr}_{-p}^F\mathrm{DR}(\mathcal{IC}_X) \arrow{r} \arrow{d}  & \mathscr  O_H \otimes_{\mathscr O_Y} \mathscr H^{-(n-p)}\mathrm{gr}_{-p}^F\mathrm{DR}(\mathcal{N}_0) \arrow{d} \\
   \mathscr H^{-(n-p)+1}\mathrm{gr}_{-p}^F\mathrm{DR}(\mathcal{IC}_{X\cap H}) \arrow{r} \arrow{d} & \mathscr H^{-(n-p)+1}\mathrm{gr}_{-p}^F\mathrm{DR}(\mathcal{N}_0|_H) \arrow{d} \\
   \mathscr H^{-(n-p)+1}\mathrm{gr}_{-p + 1}^F\mathrm{DR}(\mathcal{IC}_{X\cap H})\arrow{r} & \mathscr H^{-(n-p)+1}\mathrm{gr}_{-p + 1}^F\mathrm{DR}(\mathcal{N}_0|_H)
\end{tikzcd}\] coming from the exact triangle associated to the restricted Hodge module sequence (\ref{equation conormal cutdown}).  Again, $\mathcal{IC}_{X\cap H}$ (resp. $\mathcal N_0|_H$) is the intersection $\mathscr D_Y$-module (resp. logarithmic $\mathscr D_Y$-module) of $X\cap H$.  By (\ref{equation holomorphic forms intersection complex}) and (\ref{equation logarithmic forms logarithmic hodge module}), the first three rows of the commutative diagram agree with (\ref{equation conormal form diagram}).

We note that if $X$ has isolated singularities, the claim is known to hold \cite[Thm. (1.3)]{van1985extendability}.  Continuing by induction, consider the above diagram for $p = \mathrm{codim}_X(\Sigma) - 1$. The first horizontal morphism is an isomorphism by Flenner's criterion, and the third horizontal morphism is an isomorphism by induction.  For this $p$, we also have $$\mathscr H^{-(n-p) + 1} \mathrm{gr}_{-p+1}^F \mathrm{DR}(\mathcal{IC}_{X\cap H}) = \mathscr H^{-\dim \Sigma}\mathrm{gr}^F_{-((n-1) - \dim \Sigma)}\mathrm{DR}(\mathcal{IC}_{X\cap H}).$$ For dimension reasons, this term vanishes, and we see the map $$\mathscr O_H \otimes_{\mathscr O_Y} \pi_*\Omega_{\widetilde X}^p \to (\pi|_H)_*\Omega_{\widetilde H}^p$$ is surjective. 

Next, we follow an argument in \cite{graf2013optimal}.  Let $E = \sum E_i$ be the exceptional divisor of $\pi$, and let $\alpha$ be section in $H^0(\widetilde X,\Omega_{\widetilde X}^p(\log E))$.  By definition, there are integers $m_i \in \left\{0,1\right\}$ such that $\gamma$ is a section in $H^0(\widetilde X, \Omega_{\widetilde X}^p(\sum m_i E_i))$.  Consider the diagram \[ \begin{tikzcd}
    N_{H|Y}^*\otimes_{\mathscr O_H} (\pi|_H)_*\Omega_{\widetilde H}^{p-1} \arrow{d} \arrow{r} & \mathscr O_H \otimes \pi_*\Omega_{\widetilde X}^p \arrow{r} \arrow{d}{\alpha} & (\pi|_H)_*\Omega_{\widetilde H}^p \arrow{d} \\
    N_{H|Y}^*\otimes_{\mathscr O_Y} \pi_*(\Omega_{\widetilde H}^{p-1}\otimes \mathscr O(\sum m_i E_i)) \arrow{r} & \mathscr O_H \otimes \pi_*\Omega_{\widetilde X}^p(\sum m_iE_i) \arrow{r}{\beta} & (\pi|_H)_*(\Omega_{\widetilde H}^p\otimes \mathscr O(\sum m_iE_i)).
\end{tikzcd}\] We assume again the assumption is true in dimension $\dim X - 1$.  In particular, the image $\beta(\alpha(\gamma))$ factors through $H^0(\widetilde H, \Omega_{\widetilde H}^p)$.  Supposing by contradiction that the $m_i$ are not all 0, this implies the existence of a nonzero section of $H^0(\widetilde H, N_{\widetilde H|\widetilde X}^*\otimes \Omega_{\widetilde X}^{p-1}\otimes(\sum m_iE_i))$.  Since the pullback of $N_{H|X}$ is just $N_{\widetilde H|\widetilde X}$, this implies $H^0(\widetilde H, \Omega_{\widetilde H}^{p-1}\otimes \mathscr O(\sum m_iE_i))$. By induction and our choice of $p$, it is clear the morphism $$(\pi|_H)_*\Omega_{\widetilde H}^{p-1} \to (\pi|_H)_*(\Omega_{\widetilde H}^{p-1}\otimes \mathscr O(\sum m_iE_i))$$ is surjective.  We have also seen the morphism $\mathscr O_H \otimes \pi_*\Omega_{\widetilde X}^{p-1} \to (\pi|_H)_*\Omega_{\widetilde H}^{p-1}$ is surjective for our choice of $p$.  This produces a non-zero section of $H^0(\widetilde H, \Omega_{\widetilde H}^{p-1}(\sum m_iE_i)|_{\widetilde H})$ which vanishes under $\beta$.  We can continue to iterate this process until we receive a non-zero section of $H^0(\widetilde H, \mathscr O_{\widetilde H}(\sum m_iE_i))$.  In fact, we can find for each $E_j$ with $m_j = 1$ that $$H^0(\widetilde H, \mathscr O(m_iE_i)|_{E_j}) \ne 0.$$ This is a contradiction to the negativity lemma \cite[Proposition 7.5]{greb2010extension}. 
\qed \newline 

This gives a different proof of Theorem \ref{theorem holomorphic extension intro} for Du Bois singularities by Theorem \ref{theorem logarithmic extension intro}.  We emphasize that this proof is weaker, however, since it does not imply support condition (\ref{equation dim support proof 1.2}).  

\section{Holomorphic Extension for (Weakly) $m$-Du Bois Singularities}

\subsection{Rational v.s. Weakly $m$-Du Bois Singularities} \label{subsection rational vs weakly}

A fundamental aspect of the Hodge theory of singularities is the relationship between the Du Bois property and the singularities of the MMP: Du Bois singularities are very close to log-canonical singularities, as both agree in the normal, quasi-Gorenstein case.  The gap between rational and Du Bois singularities is much larger.  Recall by Kempf's criterion that a variety has rational singularities if and only if $X$ is Cohen-Macaulay $\pi_*\omega_{\tilde X}$ is reflexive.  We have already seen a Du Bois singularity failing holomorphic extension, but they can also fail the Cohen-Macaulay property: the affine cone of a compact hyperk\"ahler manifold in dimension $ \ge 4$ is Du Bois but not Cohen-Macaulay.  

Therefore, it seems interesting to ask how closely (weakly) $m$-Du Bois singularities are from having rational singularities.  Even for $k > 0$, there is a disconnect.  The following affine cone examples were described in \cite{tighe2023thesis}, see also \cite[\S 7.5]{shen2023k}.

\begin{example} \label{example 1-Du Bois not Rational}
        Let $Y$ be a projective K3 surface.  For any ample bundle $L$, the affine cone $X$ is Cohen-Macaulay and Du Bois by Kodaira vanishing.  Suppose further that $Y$ has Picard rank 1 and the degree $L^2 \ge 24$.  Then the affine cone is also weakly 1-Du Bois, and $\mathscr H^0\underline \Omega_X^1$ is reflexive.  Indeed, it is sufficient to show that $H^1(Y, \Omega_Y^1\otimes L^m) = 0$ for every $m > 0$.  This follows from \cite[Theorem 3.2 and Theorem 3.5]{totaro2020bott}. The first theorem states that, under these assumptions, the vanishing holds for $H^1(Y, \Omega_Y^1\otimes L) = 0$, considering the pair $(Y,L)$ as a polarized K3 surface.  The second theorem verifies the vanishing $H^1(Y, \Omega_Y^1\otimes L^m) = 0$ for $m > 1$ assuming the first vanishing.  Since $H^0(Y, \Omega_Y^1) = 0$, this implies that $X$ is weakly 1-Du Bois.  But $H^0(Y, \omega_Y) \cong \mathbb C$, and so $X$ does not have rational singularities. By Theorem \ref{theorem holomorphic extension intro}, $\pi_*\Omega_{\widetilde X}^2$ is reflexive for any resolution of singularities.   
        \end{example}

        \begin{example} \label{example Hypersurface Cone Rational}
        Here is an example of an affine cone singularity which is rational but is not $m$-Du Bois for some $1 \le k \le n-1$.  The example is essentially given in \cite[\S 4.1]{buch1997frobenius} and is related to the failure of Bott vanishing for non-projective spaces.  Let $Y \subset \mathbb P^4$ be a smooth quadric hypersurface.  Let $L = \mathscr O_Y(1)$.  Then the affine cone $X$ has rational singularities since $Y$ is a Fano variety.  The cone is also weakly 1-Du Bois.  To see this\footnote{Since $Y$ is a hypersurface, then so will the affine cone $X$.  Therefore this follows from \cite{mustata2021bois}}, we consider the cohomology sequence of $$0 \to \mathscr O_Y(2 + m) \to \Omega_{\mathbb P^4}^1\otimes \mathscr O(m)\otimes \mathscr O_Y \to \Omega_Y^1\otimes \mathscr O_Y(m) \to 0.$$  The vanishing $H^i(Y, \Omega_Y^1\otimes \mathscr O_Y(m)) = 0$ follows then by Bott vanishing on $\mathbb P^4$ and the rationality of $(X,v)$.  On the other hand, $H^1(Y, \Omega_Y^2\otimes \mathscr O_Y(1)) \cong \mathbb C$.  This implies that $X$ is not weakly $2$-Du Bois.
        \end{example}

In summary: weakly $m$-Du Bois singularities need not have rational singularities, nor are rational singularities weakly $m$-Du Bois for all $k$.  What is interesting about Example \ref{example 1-Du Bois not Rational} is that the 3-fold singularity is weakly $3$-Du Bois, and yet holomorphic extension does not hold in degree 3.  This is because $\mathscr H^0\underline \Omega_X^2$ is also not reflexive in this case, a necessary condition for holomorphic extension to hold for $m$-Du Bois singularities.

\subsection{Isolated Singularities and Depth}\label{subsection isolated singularities and depth} To further highlight what Example \ref{example 1-Du Bois not Rational} tells us, we consider the relationship between $\mathscr H^0\underline \Omega_X^p$, $\pi_*\Omega_{\widetilde X}^p$, and $\mathrm{depth}(\mathscr O_X)$.  Let $X$ be a normal variety with isolated singularities.  If $X$ is weakly $m$-Du Bois, then Theorem \ref{theorem Du Bois properties}\ref{theorem right triangle resolution exceptional divisor} implies a short exact sequence $$0 \to \mathscr H^0\underline \Omega_X^p \to \mathscr \pi_*\Omega_{\tilde X}^p \to \pi_*\underline \Omega_E^p \to 0$$ for $0 \le p \le k$.  An immediate consequence of Theorem \ref{theorem Du Bois properties}\ref{theorem Du Bois open restriction} is that $\pi_*\underline \Omega_E^p = 0$ if $\mathscr H^0\underline \Omega_X^p$ is reflexive.  We note that $\underline \Omega_E^p = \Omega_E^p/\mathrm{tor}$ is just a sheaf; if $E$ is an snc divisor, then $H^0(E, \underline \Omega_E^p)$ vanishes by Hodge theory if $H^p(E, \mathscr O_E) = 0$.  This for example holds if $\mathrm{depth}(\mathscr O_X) \ge p + 2$ (compare with \cite[Theorem G]{mustata2021hodge}).

Note that even if we assume $X$ is Cohen-Macaulay that we do not get the vanishing of $\pi_*\underline \Omega_E^{\dim X-1}$: this is precisely what happens in Example \ref{example 1-Du Bois not Rational}.  The reflexivity of $\mathscr H^0\underline \Omega_X^{\dim X-1}$ appears independent of any condition on the depth, similar to the sheaf $\pi_*\omega_{\widetilde X}$.  

\subsection{A criterion for Holomorphic Extension for Weakly $m$-Du Bois Singularities}

To summarize, the weaker version of the higher Du Bois property does not appear to detect the holomorphic extension property.  However, we can remedy this by requiring $\mathscr H^0\underline \Omega_X^k$ to be reflexive.  The following is an extension of Theorem \ref{theorem holomorphic extension intro} for weakly $m$-Du Bois singularities.

\begin{theorem} \label{theorem holomorphic extension criterion k Du Bois}
Let $X$ be a normal complex algebraic variety of dimension $n$ with singular locus $\Sigma$.  Suppose $X$ is weakly $m$-Du Bois, where $k = \mathrm{codim}_X(\Sigma)-1$.  If $\mathscr H^0\underline \Omega_X^k$ is reflexive, then $\pi_*\Omega_{\widetilde X}^{k+1}$ is reflexive. 
\end{theorem}

\begin{proof}
For isolated singularities, we are checking $\pi_*\omega_{\widetilde X}$ is reflexive when $X$ is weakly $(n-1)$-Du Bois and $\mathscr H^0\underline \Omega_X^{n-1}$ is reflexive.  By assumption, the quasi-isomorphism $$\mathscr H^0\underline \Omega_X^{n-1}\cong \mathbf R\pi_*\Omega_{\widetilde X}^{n-1}(\log E)(-E)$$ implies $\dim \mathrm{Supp}~R^j\pi_*\Omega_{\widetilde X}^1(\log E) \le n-j-2$.  In particular, $R^{n-1}\pi_*\Omega_{\widetilde X}^1(\log E) = 0$.  This gives the additional vanishing $R^{n-1}\pi_*\Omega_{E_{(1)}}^1 = 0$ by the residue exact sequence.  We may pass to cohomology by shrinking $X$ as necessary; by Hodge theory, this gives $\pi_*\Omega_{E_{(1)}}^{n-1} = 0$.  This implies $\pi_*\omega_{\widetilde X} \hookrightarrow \pi_*\omega_{\widetilde X}(E)$ is an isomorphism. Since $X$ is assumed to be Du Bois, this proves the result by Theorem \ref{theorem logarithmic extension intro}.  

In order to prove $\pi_*\Omega_{\widetilde X}^{\mathrm{codim}_X(\Sigma)}$ is reflexive, it is sufficient to prove the stronger claim $$\dim \mathrm{Supp}~\mathscr H^j\mathrm{gr}_{-\dim \Sigma}^F\mathrm{DR}(\mathcal{IC}_X) \le -(j+2).$$ For $k = n-1$ (i.e., isolated singularities), this is equivalent to $\pi_*\omega_{\widetilde X}$ being reflexive.  We can therefore proceed by induction as in the proof of Theorem \ref{theorem log forms to holomorphic forms}.  Specifically, the higher Du Bois property is preserved by general hyperplane \cite[Theorem A]{shen2023k}, and the reflexivity of $H^0\underline \Omega_X^{\mathrm{codim}_X(\Sigma)}$ is also preserved by hyperplane.  As a result, we get the desired support condition except possibly when $j = -1$; but this is only relevant when $X$ has isolated singularities. 
\end{proof}

\begin{corollary} \label{corollary dim support condition k Du Bois rational}
Let $X$ be a normal complex variety with weakly $m$-Du Bois singularities for $k \ge \mathrm{codim}_X(\Sigma) - 1$, where $\Sigma$ is the singular locus of $X$.  If $\mathscr H^0\underline \Omega_X^p$ is reflexive for each $0 \le p \le k$, then \begin{equation}\label{equation dim support rational complex}
   \dim\mathrm{Supp}~R^j\pi_*\mathscr O_{\widetilde X} \le n-j-2 
\end{equation} for each $j \le k$.
\end{corollary}

\begin{proof}
For isolated singularities, the assumption on $k$ implies $\pi_*\omega_{\widetilde X}$ is reflexive by Theorem \ref{theorem holomorphic extension criterion k Du Bois}.  Therefore the claim (\ref{equation dim support rational complex}) holds by Proposition \ref{proposition kebekus sheaf criterion} and Grauert-Riemenschneider vanishing.  More generally, let $X \subset Y$ be a locally closed embedding into a smooth manifold $Y$ and let $H$ be a general hyperplane section of $Y$.  By Theorem \ref{theorem Du Bois properties}\ref{theorem Du Bois base point free}, $X\cap H$ has weakly $m$-Du Bois singularities, and it is clear that $\mathscr H^0\underline \Omega_{X\cap H}^p$ is reflexive whenever $\mathscr H^0\underline \Omega_X^p$ is reflexive.  Therefore induction and the isomorphism $$\mathscr O_H\otimes \mathbf R\pi_*\mathscr O_{\tilde X} \cong \mathbf R(\pi|_H)_*\mathscr O_{\widetilde{X\cap H}}[1],$$ where $\pi|_H$ is the induced resolution of singularities of $X\cap H$ from $\pi:\widetilde X \to X$, imply the claim except in the case $j = n-1$; but this is only relevant when $X$ is $(n-1)$-Du Bois and so follows from Theorem \ref{theorem holomorphic extension criterion k Du Bois}.     
\end{proof}  

\subsection{A Remark on the Functorial Pullback Morphism}\label{subsection functorial pullback}

Corollary \ref{corollary dim support condition k Du Bois rational} is optimal for weakly $m$-Du Bois singularities.  If $X$ is Cohen-Macaulay, then $\mathscr H^0\underline \Omega_X^p$ is reflexive for $p \le \mathrm{codim}_X(\Sigma) - 2$.  This follows since:

\begin{itemize}
    \item $R^p\pi_*\mathscr O_{\widetilde X} = 0$ for $p \le \mathrm{codim}_X(\Sigma) - 2$ \cite[Lemma 3.3]{kovacs1999rational}, and

    \item The holomorphic extension property holds for $p \le \mathrm{codim}_X(\Sigma) - 2$ by Flenner's criterion.
\end{itemize} The inclusion $\mathscr H^0\underline \Omega_X^p \hookrightarrow \pi_*\Omega_{\widetilde X}^p$ is therefore an isomorphism due to an idea of Kebekus-Schnell on the existence of functorial pullback morphisms.

Recall that if $f:Z \to X$ is a morphism of complex spaces, there is a functorial pullback morphism $f^*\Omega_X^p \to \Omega_Z^p$ between the sheaves of K\"ahler $p$-forms.  In general, this pullback does not extend to the sheaves of reflexive differentials.  Work of Kebekus \cite{kebekuspullback} describes a process of constructing a natural reflexive pullback morphism which agrees with the K\"ahler pullback morphism on the regular locus.  The pullback morphism was originally constructed for morphisms of algebraic varieties with at worst klt singularities but was extended to arbitrary complex spaces with rational singularities \cite[\S 14]{kebekus2021extending}.  There are two major inputs for a pullback morphism $f^*\Omega_X^{[p]} \to \Omega_Z^{[p]}$ to exist in degree $p$, which hold for rational singularities:

\begin{itemize}
    \item The varieties $X$ and $Z$ satisfy the holomorphic extension property in degree $p$ \cite[Corollary 1.8]{kebekus2021extending}

    \item If $E_t = \pi^{-1}(t)$ is the fiber of a resolution of singularities $\pi:\tilde X \to X$, then $H^0(E_t, \Omega_{E_t}^p/\mathrm{tor}) = 0$ \cite[Lemma 1.2]{namikawa2000deformation}.
\end{itemize}

Recall from \S\ref{subsection isolated singularities and depth} that the second condition holds for isolated weakly $m$-Du Bois singularities whenever $\mathscr H^0\underline \Omega_X^p$ is reflexive for $p \le k$.  The first condition also holds in degree $p  =k + 1$ by Theorem \ref{theorem holomorphic extension criterion k Du Bois}.  Since $\mathscr H^0\underline \Omega_X^p$ agrees with the sheafification of $\Omega_X^p$ in the $h$-topology \cite{huber14}, $\mathscr H^0\underline \Omega_X^p$ is reflexive whenever these conditions for functorial pullback hold.  The following corollary of Theorem \ref{theorem holomorphic extension criterion k Du Bois} is therefore immediate.

\begin{corollary} \label{corollary extension criterion divisor}
Let $X$ be a normal complex variety with singular locus $\Sigma$, and let $\pi:\widetilde X \to X$ be a log-resolution of singularities.  Suppose $X$ is weakly $m$-Du Bois for $k \ge \mathrm{codim}_X(\Sigma) - 1$ and that $\mathscr H^0\underline \Omega_X^p$ is reflexive for some $p \le k$.  Then $\mathscr H^0\underline \Omega_X^{p+1}$ is reflexive if and only if the fibers $H^0(E_t, \Omega_{E_t}^{p+1}/\mathrm{tor}) =0$.
\end{corollary}

\subsection{A Remark on the Reflexivity of $\Omega_X^p$} \label{section Kahler forms} Let $X$ be a normal complex algebraic variety.  We wish to consider the $m$-Du Bois and $m$-rational definitions defined in \cite{mustata2021bois}, \cite{jung2021higher}, \cite{mustata2021hodge}, \cite{friedman2022deformations}, and \cite{friedman2022higher}.  These papers consider lci singularities, in which case the $m$-Du Bois properties implies $\Omega_X^p$ is reflexive for every $p$; this is particularly special for lci singularities and restricts the codimension of the singular locus.  For instance, if $Y$ is a variety with lci singularities, $\Omega_Y^1$ is reflexive if and only if $Y$ is smooth in codimension 3 \cite{kunz1986kahler}.

Understanding what happens when $\Omega_X^p$ is reflexive seems to be a difficult problem and rather restrictive.  One reason for this is that the minimal generating sets of the sheaves $\Omega_X^k$ are related to the embedding dimension of $X$; in particular, a minimal generating set of the $\mathscr O_{X,x}$-module $\Omega_{X,x}^k$ has $\binom{e_x}{k}$ generators, where $e_x$ is the embedding dimension of the singularity $(X,x)$ \cite[\S 4]{graf2015generalized}.  In many cases, the reflexivity of $\Omega_X^k$ necessarily restricts the embedding dimension.  The reflexivity of the sheaf of K\"ahler differentials seems to be particularly special: the only known example of a singular variety with $\Omega_X^p$ reflexive for some $p \ge 1$ are locally complete intersections.  Even in this case, the higher K\"ahler $p$-forms will contain torsion and cotorsion \cite[Theorem 1.11]{graf2015generalized}. Beyond this, little is known about the reflexivity of the sheaves of K\"ahler $p$-forms in general.  We give one new result in this direction:

\begin{proposition}
If $X$ is an $n$-fold Gorenstein variety with at worst quotient singularities, then $\Omega_X^p$ is reflexive for some $1 \le p \le n$ if and only if $X$ is smooth.
\end{proposition}

\begin{proof}
We use a description of the sheaves $\pi_*\Omega_{\widetilde X}^p$ used in \cite[\S 10]{kebekus2021extending}.  Let $X \subset Y$ be an embedding of $X$ into a smooth complex manifold $X$ of codimension $c$.  Let $\mathcal{IC}_X$ be the $\mathscr D_Y$-module associated to the intersection Hodge module.  The graded components of the de Rham complex with respect to the Hodge filtration are of the form $$\mathrm{gr}_{-p}^F\mathrm{DR}(\mathcal{IC}_X) = [\Omega_Y^{p+c}\otimes F_c\mathcal{IC}_X \xrightarrow{\nabla} \Omega_Y^{p+1+c}\otimes \mathrm{gr}_{c+1}^F\mathcal{IC}_X \xrightarrow{\nabla}... \xrightarrow{\nabla} \Omega_Y^{n+c}\otimes \mathrm{gr}_{n-p + c}^F\mathcal{IC}_X]$$ shifted by degree $-(n-p)$.  Therefore, we have $$\Omega_X^{[p]} \cong \mathrm{ker}(\Omega_Y^{p+c}\otimes F_c\mathcal{IC}_X \xrightarrow{\nabla} \Omega_Y^{p+1+c}\otimes \mathrm{gr}_{c+1}^F\mathcal{IC}_X)$$ whenever $X$ has rational singularities by (\ref{equation holomorphic forms intersection complex}). 

Now we use the assumption $X$ has at worst quotient singularities.  Since $X$ is a rational homology manifold, the natural morphisms $$\underline \Omega_X^p \to \mathrm{gr}_{-p}^F\mathrm{DR}(\mathcal{IC}_X)[p-n] \to \mathbb D_X(\underline \Omega_X^{n-p})$$ are quasi-isomorphisms for every $p$; in fact, there is an isomorphism $$\mathbb Q_X[n] \xrightarrow{\sim} IC_X$$ of \textit{perverse} sheaves, and $\mathrm{gr}_{k}^F\mathcal{IC}_X = 0$ for $k \ge c + 1$.  The above description of $\mathrm{DR}(\mathcal{IC}_X)$ implies \begin{equation}\label{equation reflexive isomorphism smooth embedding}
    \Omega_X^{[p]} \cong \Omega_Y^{p+c}\otimes F_c\mathcal{IC}_X.
\end{equation} If we assume $\Omega_X^p$ is reflexive, then $$\Omega_X^p \cong \Omega_Y^{p+c}\otimes F_c\mathcal{IC}_X.$$  By Claim 10.2 and the identification (\ref{equation reflexive isomorphism smooth embedding}), this isomorphism is closed under wedging with K\"ahler forms on $Y$.  Specifically, wedging $\Omega_Y^{p+c}\otimes F_c\mathcal{IC}_X$ with K\"ahler $(n-p)$ forms on $Y$ lands in $\Omega_Y^{n+c}\otimes F_c\mathcal{IC}_X \cong \omega_X$.  By restricting to $X$, we see that $$\Omega_X^n = \Omega_X^p \wedge \Omega_X^{n-p} \subset \omega_X.$$  This implies $\Omega_X^n$ is torsion-free.  On the other hand, since we assume $\omega_X$ is a line bundle, $\Omega_X^n \to \omega_X$ is surjective.  Since $\omega_X = (\Omega_X^n)^{**}$, this implies $\Omega_X^n \xrightarrow{\sim} \omega_X$, and so $\Omega_X^n$ is locally free.  This means $X$ is smooth.
\end{proof}

\section{Hodge Theory of Proper Varieties with Weakly $m$-rational Singularities}

\subsection{Mixed Hodge Theory for MMP Singularities} Let $X$ be a complex algebraic variety.  In his thesis, Deligne defined the notion of a mixed Hodge structure and constructed a canonical mixed Hodge structure on the cohomology of any complex algebraic variety \cite{deligne1971theorie}, \cite{deligne1974theorie}.  Specfically, there is an increasing weight filtration $W_\bullet$ on $H^k(X, \mathbb Q)$ and a decreasing filtration $F^\bullet$ on $H^k(X, \mathbb C)$ which descends to a Hodge filtration on $\mathrm{gr}_p^WH^k(X, \mathbb Q)\otimes \mathbb C$ for each $k$.  The Du Bois complex is a byproduct of this construction: namely, Du Bois details in \cite{du1981complexe} that Deligne's construction produces a complex $\underline \Omega_X^\bullet$ for any algebraic variety $X$ which generates the Hodge filtration under the spectral sequence  $$E_1^{p,q} = \mathbb H^q(X, \underline \Omega_X^p) \Rightarrow H^{p+q}(X, \mathbb C)$$ when $X$ is proper.

It is often useful to understand when the mixed Hodge structure on $H^k(X, \mathbb Q)$ is pure when studying the global moduli of singularities.  For low degree, this is usually understood by looking at the pullback morphism $\pi^*:\pi^*H^k(X, \mathbb Q) \to H^k(\widetilde X, \mathbb Q)$ associated to a resolution of singularities.  By Leray, this map is injective if $k = 1$ when $X$ is normal.  For $k = 2$, this map is again injective by Leray when $X$ has rational singularities.

The obvious question to ask is how does the (weakly) $m$-Du Bois property affect the mixed Hodge theory on the cohomology $H^*(X, \mathbb Q)$ of a projective variety.  Surprisingly, the vanishing cohomology of the $\underline \Omega_X^p$ has little control over the weight filtration:

\begin{example}
    \begin{enumerate}
        \item Let $X$ be a projective curve.  For dimension reasons, $X$ is weakly $1$-Du Bois, but $H^1(X, \mathbb Q)$ will not carry a pure Hodge structure.  For an explicit example, consider the nodal elliptic curve.

        \item Let $X$ be a projective hypersurface of dimension 3 with ordinary double points.  Note that $X$ has rational singularities and is $1$-Du Bois by \cite[Corollary 1.9]{friedman2022higher}.  Therefore, $H^2(X, \mathbb Q)$ carries a pure Hodge structure by the preceding discussion, but $H^3(X, \mathbb Q)$ need not carry a pure Hodge structure

        \item Let $Y$ be a projective hyperk\"ahler 4-fold manifold.  Suppose there exists a birational contraction $\phi:Y \to X$ of a Lagrangian submanifold $L \cong \mathbb P^2 \subset Y$ to a point.  Then $X$ is \textit{not} weakly $1$-rational \cite[Proposition 1.4]{tighe2024symmetries}, but $H^k(X, \mathbb Q)$ carries a pure Hodge structure for each $k$.
    \end{enumerate}
\end{example}

Moreover, $H^k(X, \mathbb Q)$ will carry a pure Hodge structure for large $k$ for topological reasons: if $X$ is projective and has isolated singularities, then $H^k(X, \mathbb Q)$ carries a pure Hodge structure of weight $k$ for $k > n$.  What is interesting in this case is that dual $H^k(X, \mathbb Q(-n))^*$ carries a pure Hodge structure of weight $2n-k$.  By Poincar\'e duality, this group is $H^{2n-k}(X_{\mathrm{reg}}, \mathbb Q)$, and $H^{2n-k}(X, \mathbb Q)$ carries a pure Hodge structure if and only if $H^{2n-k}(X, \mathbb Q) \to H^{2n-k}(X_{\mathrm{reg}}, \mathbb Q)$ is injective.

More generally, we can consider the weight filtrations for the mixed Hodge structures on $H^k(X, \mathbb Q)$ and $H^{2n-k}(X, \mathbb Q(-n))^*$, respectively.  Since $X$ is proper, the weight filtration $W_\bullet$ truncates to $H^k(X, \mathbb Q)$ for each $k$.  The weight filtration on the dual $H^{2n-k}(X, \mathbb Q(-n))^*$ therefore is supported in higher weights (compare this to the cohomology $H^k(X_{\mathrm{reg}}, \mathbb Q)$), and these groups are related to the Grothendieck duals $\mathbb D_X(\underline \Omega_X^{n-p})$ of the Du Bois complex.  It therefore seems better to consider how the $m$-rational property affects the weight filtration.

\subsection{Purity for $m$-rational Singularities}

\begin{theorem} \label{theorem k-rational hodge}
If $X$ is a normal and projective variety of dimension $n$ with weakly $m$-rational singularities, then the canonical mixed Hodge structure on $H^m(X, \mathbb Q)$ is pure of weight $m$.
\end{theorem}

\begin{proof}
By \cite[Theorem B]{shen2023k}, we have $$\Omega_X^{[p]} \cong_{\mathrm{qis}} \underline \Omega_X^p \cong \mathbb D_X(\underline \Omega_X^{n-p})$$ for every $0 \le p \le m$.

On the one hand, the spectral sequence $$E_1^{p,q} = \mathbb H^q(X, \underline \Omega_X^p) \Rightarrow H^{p+q}(X, \mathbb C)$$ generates the Hodge filtration.  Note that the weight filtration $W_{\mathbb C}^\bullet$ is supported in weight $\le n$, since $X$ is proper.

On the other hand, $\mathbb H^q(X, \mathbb D_X(\underline \Omega_X^{n-p})) \cong \mathrm{Hom}_{\mathscr O_{\mathrm{pt}}}(\mathbb H^{n-q}(X,\underline \Omega_X^{n-p}), \mathscr O_{\mathrm{pt}})$ by duality.  Therefore $E_1^{p,q}$ generates the Hodge filtration on the dual mixed Hodge structure $H^{n-m}(X, \mathbb C)^*$ for $m = p + q$.  The weight filtration on this mixed Hodge structure is supported in weight $\ge n$. Therefore, $H^m(X, \mathbb Q)$ must be a pure Hodge structure.  
\end{proof}

\begin{corollary}
If $X$ is a projective variety of dimension $n$ with weakly $m$-rational singularities, then there is a non-canonical decomposition $$H^m(X, \mathbb C) = \bigoplus_{p + q = m} H^q(X, \pi_*\Omega_{\tilde X}^p)$$ induced by the Hodge filtration.
\end{corollary}

We remark that Theorem \ref{theorem k-rational hodge} implies something stronger than the cohomology $H^m(X, \mathbb Q)$ carrying a pure Hodge structure.  For instance if $m < \mathrm{codim}_X(\Sigma)$, where $\Sigma$ is the singular locus, the above proof shows that $H^m(X, \mathbb Q) \to H^m(X_{\mathrm{reg}}, \mathbb Q)$ is an isomorphism.  If $X$ has rational singularities, this morphism is always injective for $m =2$.  If $X$ is a 3-fold with isolated singularities, then $H^m(X, \mathbb Q) \to H^m(X_{\mathrm{reg}}, \mathbb Q)$ is an isomorphism if and only if Poincar\'e duality holds: therefore the \textit{defect} $\sigma(X)$ of $X$ must be 0 \cite[p. 97]{kawamata88crepant}, and $X$ is $\mathbb Q$-factorial.  This extends to higher dimensions in special cases, see for example \cite[Proposition 2.17]{tighe2022llv}.

\bibliography{bib.bib}
\bibliographystyle{alpha}

\end{document}